\documentclass[11pt]{amsart}
\usepackage{graphicx} 
\title{Operad on the free monoid on species}
\author{}
\date{October 2023}
\usepackage{axodraw4j}
\usepackage[utf8]{inputenc}
\usepackage{amssymb}
\usepackage{graphicx}
\usepackage{color}
\usepackage{amsmath}
\usepackage{amsfonts}
\usepackage{geometry}
\usepackage{graphicx}
\usepackage{mathrsfs}
\usepackage{amscd}
\usepackage{epic}
\usepackage{shuffle}[2008/10/27 shuffle product symbol]
\usepackage{xcolor}
\definecolor{ForestGreen}{RGB}{34,139,34}
\usepackage[colorinlistoftodos,textsize=tiny]{todonotes}
\usepackage{xargs} 
\usepackage[all]{xy}
\usepackage{tikz-cd}
\usepackage{comment}
\usetikzlibrary{arrows.meta}
\theoremstyle{plain}

\newtheorem{theorem}{Theorem}[section]

\newtheorem{prop}[theorem]{Proposition}
\newtheorem{lem}[theorem]{Lemma}

\newtheorem{definition}[theorem]{Definition}

\newtheorem{lemma}[theorem]{Lemma}
\newtheorem{remark}[theorem]{Remark}
\newtheorem{proposition}[theorem]{Proposition}

\newtheorem{notation}[theorem]{Notation}

\theoremstyle{definition}

\newtheorem{example}[theorem]{Example}

\numberwithin{equation}{section}

\font \eightrm=cmr8

\newcommand{\mop}[1]{\mathop{\hbox {\rm #1} }}
\newcommand{\smop}[1]{\mathop{\hbox {\eightrm #1} }}

\long\def\ignore#1{}

\def \srestr#1{\mathstrut_{\scriptstyle |}\hbox to
-1.5pt{}\raise-4pt\hbox{$\hskip 1pt\scriptscriptstyle #1$}}

\def\shuffle{\sqcup\mathchoice{\mkern-7mu}{\mkern-7mu}{\mkern-3.2mu}{\mkern-3.8mu}\sqcup}

\newcommand{\sh}{{{\scriptstyle{\shuffle}}}}

\def\L1#1{L^1(#1)}

\def\L#1#2{L^{#1}(#2)}

\newcommand{\Ss}{\mathsf{Sp}} 

\newcommand{\tp}{\mathbf{p}} 
\newcommand{\tq}{\mathbf{q}}

\newcommand{\taa}{\mathbf{a}} 
 
\newcommand{\tc}{\mathbf{c}}

\newcommand{\trr}{\mathbf{r}} 
\newcommand{\tm}{\mathbf{m}} 

\newcommand{\sq}{{\scriptstyle\square}}

\newcommand{\wE}{\mathbf{E}} 
\newcommand{\wL}{\mathbf{L}} 
\newcommand{\tPi}{\mathbf{\Pi}} 

\usepackage[colorinlistoftodos,textsize=tiny]{todonotes}
\newcommandx{\unsure}[2][1=]{\todo[linecolor=red,backgroundcolor=red!25,bordercolor=red,#1]{#2}}
\def\lef({\left(}
\def\rig){\right)}

\def\preoperad{
 \begin{picture}(709,194) (31,-159)
    \SetWidth{1.0}
    \SetColor{Black}
    \Line(96,-30)(64,34)
    \Line(96,-30)(96,34)
    \Line(96,-30)(128,34)
    \Line(96,-30)(112,-94)
    \Line(112,-94)(128,-30)
    \Line(48,-94)(96,-158)
    \Line(80,-94)(96,-158)
    \Line(112,-94)(96,-158)
    \Line(144,-94)(96,-158)
    \Line(48,-94)(96,-158)
    \Line(48,-94)(96,-158)
    \Line(320,-94)(272,-158)
    \Line(288,-94)(272,-158)
    \Line(256,-94)(272,-158)
    \Line(224,-94)(272,-158)
    \Line(288,-94)(304,-30)
    \Line(272,-30)(288,-94)
    \Line(272,-30)(304,34)
    \Line(272,-30)(272,34)
    \Line(272,-30)(240,34)
    \Line(480,-94)(496,-158)
    \Line(704,-30)(656,-94)
    \Line(496,-30)(480,-94)
    \Line(464,-30)(480,-94)
    \Line(432,-30)(480,-94)
    \Line(496,-30)(528,34)
    \Line(496,-30)(496,34)
    \Line(496,-30)(464,34)
    \Line(672,-158)(688,-94)
    \Line(656,-94)(672,-158)
    \Line(496,-158)(512,-94)
    \Line(528,-30)(480,-94)
    \Line(672,-30)(656,-94)
    \Line(640,-30)(656,-94)
    \Line(608,-30)(656,-94)
    \Line(672,-30)(704,34)
    \Line(672,-30)(672,34)
    \Line(672,-30)(640,34)
    \Bezier[dash,dsize=10](32,-78)(96,-94)(128,-94)(176,-78)
    \Bezier[dash,dsize=10](224,-14)(256,-30)(288,-30)(320,-14)
    \Bezier[dash,dsize=10](400,-78)(464,-94)(496,-94)(544,-78)
    \Bezier[dash,dsize=10](592,-14)(656,-30)(688,-30)(736,-14)
    \Line(160,-126)(192,-126)
    \Line(368,-126)(400,-126)
    \Line(368,-142)(400,-142)
    \Line(560,-126)(592,-126)
    \Text(92,-105)[lb]{\Huge{\Black{$s$}}}
    \Text(77,-45)[lb]{\Huge{\Black{$t$}}}
    \Text(267,-105)[lb]{\Huge{\Black{$s$}}}
    \Text(252,-45)[lb]{\Huge{\Black{$t$}}}
    \Text(464,-110)[lb]{\Huge{\Black{$t$}}}
    \Text(478,-40)[lb]{\Huge{\Black{$s$}}}
    \Text(640,-110)[lb]{\Huge{\Black{$t$}}}
    \Text(655,-40)[lb]{\Huge{\Black{$s$}}}
  \end{picture}

}

\begin{document}
\title{On the free commutative monoid over a positive operad}
\author{Dominique Manchon, Hedi Regeiba, Imen Rjaiba, Yannic Vargas}
\address[D. Manchon]{Laboratoire de Math\'ematiques Blaise Pascal, CNRS - Universit\'e Clermont-Auvergne, 3 place Vasar\'ely, CS 60026, 63178 Aubi\`ere, France}
\email{Dominique.Manchon@uca.fr}
\address[H. Regeiba]{Laboratory Mathematics and Applications of Gabes, LR 17 ES 11, Erriadh City  6072, Zrig, Gabes, Faculty of sciences  of Gabes, Tunisia} \email{rejaibahedi@gmail.com}
\address[I. Rjaiba]{Laboratory Mathematics and Applications of Gabes, LR 17 ES 11, Erriadh City  6072, Zrig, Gabes, Faculty of sciences  of Gabes, Tunisia} 
\email{imen.rjaiba95@gmail.com}
\address[Y. Vargas]{CUNEF Universidad, Department of Mathematics, C. de Almansa, 101, 28040 Madrid, Spain.}
\email{yannic.vargas@cunef.edu}
\begin{abstract}
We study algebraic structures on the free commutative twisted algebra generated by a positive operad $\tq$, in the framework of vector species. Given a nonunital commutative twisted algebra structure $\mu$ on $\tq$, we introduce the notion of $\mu$-compatible operad structure, leading to a nonunital operad structure on $\wE \circ \tq$, where $\wE$ stands for the exponential species. Next, we define nested pre-Lie operads (NPL-operads), a weak form of the notion of operad, in which the nested associativity axiom is weakened down to a nested pre-Lie condition. This structure is new up to our knowledge. Several constructions of NPL-operads are presented. Finally, we define algebras over a NPL-operad, based on the notion of polynomial functions.
\end{abstract}

\maketitle{}


\section{\bf{Introduction}}\label{sec:1}
PROPs (for PROducts and Permutations) were introduced in category theory by J. F. Adams and S. Mac Lane to encode operations with several inputs and outputs. J. M. Boardman and R. M. Vogt introduced this notion in topology \cite{BV1968}. J. P. May coined the word \textsl{operad} in 1972 to qualify PROPs with only one output, and used them for understanding iterated loop spaces \cite{M1972}. Operads received a renewed interest in the Nineties of the last century, together with reformulations in purely algebraic frameworks \cite{S1986, GJ1994, Lo1996, Ma1996}. The elegant formulation in the formalism of A. Joyal's species \cite{J1981, J1986} has been particularly developed by F. Bergeron, G. Labelle and P. Leroux \cite{BLL1998}, M. M\'endez \cite{M2015} as well as by M. Aguiar and S. Mahajan \cite{AM2010, AM2020}, among others. Recent and now classical sources on operads include the monograph by J.-L. Loday and B. Vallette \cite{LV2012}, as well as the three books by B. Fresse \cite{F2009, F2017a, F2017b}.\\
  
This work is part of a general program aiming to answer the following question: given two operads $\tp$ and $\tq$ with $\tq$ positive (i.e. without component of arity zero), does any operad structure exist on their substitution $\tp \circ \tq$?\\
  
A well-known sufficient condition is the existence of a \emph{distributive law}. Indeed, if $\tp$ and $\tq$ are positive and if there exists a distributive law $\delta: \tq\circ \tp \to \tp\circ \tq$ then $(\tp \circ \tq, \gamma_{\tp \circ \tq},  \eta_{\tp}\circ \eta_{\tq})$ is an operad, where $\eta_{\tp}$ and $\eta_{\tq}$ are the respective units of $\tp$ and $\tq$ and 
\[\gamma_{\tp \circ \tq}:= (\gamma_{\tp} \circ \gamma_{\tq})\, (\text{id}_{\tp} \circ \delta \circ \text{id}_{\tq}),\]
where $\gamma_{\tp} $ and $\gamma_{\tq}$ are the multiplications of $\tp$ and $\tq$, respectively (see also Section 8.6.1 in \cite{LV2012}). Distributive laws were defined by J. Beck in 1969 for triples \cite{Beck1969}; the corresponding notion for operads can be found in \cite{Ma1996}.\\
  
The first and fourth author studied the case, probably out of the distributive law framework, where $\tp=\wL$ and where $\tq$ is a positive operad, with $\wL$ being the list operad (or \textsf{As}). They succeeded in describing by hand an operad structure on $\wL \circ \tq$ using compositions (ordered partitions) of sets and the shuffle product \cite{MV2023}. In \cite{Rjaiba2025}, the third author introduces operad structures on $\tp \circ \tq$, in the case of $\tp \in \{\textsf{Mag}, \widetilde{\textsf{Mag}},  \textsf{Nap}\}$, where \textsf{Mag} corresponds to the magmatic operad, \textsf{Nap} is the Non-Associative Permutative operad, and $\widetilde{\textsf{Mag}}$ is a new operad introduced by the third author on the species of planar rooted trees, involving a shuffle operation on the branches.\\
  
The present work deals with the study of operadic-like structures on $\wE \circ \tq$, when $\wE$ is endowed with the commutative operad structure $\textsf{Com}$, and $(\tq, \gamma_\tq)$ is any positive operad. We therefore use (unordered) set partitions instead of set compositions when working in $\wL \circ \tq$. When $\tq$ is endowed with a nonunital commutative twisted algebra structure $\mu$, compatible in some sense with $\gamma_\tq$ (see Definition \eqref{OperadCompatible}), we show how to endow $\wE \circ\tq$ with a nonunital operadic structure (Theorem \eqref{TheoremMComp}). Our assumption on the compatibility of $\gamma_\tq$ with $\mu$ is equivalent to a $\textsf{Com}_+$-\,$\tq$ - bimodule structure on $\tq$, in the category $(\Ss, \circ)$ of species with composition. We describe how the notion of operad compatible with a commutative twisted algebra appears naturally on the species $\wE_+$ and $\wL_+$. In particular, we show an apparently unknown but simple relation between the operad $\textsf{As}_+$ and the shuffle operation.\\

In the second part of this work, we introduce a new partial operation \eqref{square-partial} on $\wE\circ \tq$, similar to the pre-Lie operad on trees. This leads to a weak, pre-Lie-like form of operad (Theorem \ref{thm:main}), which we called \textsl{NPL-operad structure} (nested pre-Lie). More precisely, as symmetric operads are monoidal objects in the category of positive species with the substitution operation, NPL-operads can be seen, in some way, as \emph{pre-Lie monoidal objects} in the same monoidal category.\\

When a NPL-operad is defined on $\wE\circ \tq$, from an operad structure on a positive species $\tq$, the global NPL-structure can be described through the combinatorics of the partition lattice. Given a finite set $I$, let $\tau$, $\pi$ and $\rho$ be set partitions of $I$ such that $\tau \leq \pi \leq \rho$ are ordered through the \emph{refinement order} on set partitions. There is a unique map $f:\pi \to \tau$ that sends each block $T$ of $\pi$ to the unique block $S$ of $\tau$, such that $T\subseteq S$. In particular, the fiber of $f$ over $S \in \tau$ is the set whose elements are the blocks of $\pi$ contained in the block $S$ of $\tau$. Then consider the vector space $\tq[\tau : \pi]$ given by 
\[\tq[\tau\,:\,\pi]:= \displaystyle \bigotimes_{S\in \tau} \tq[f^{-1}(S)].\]
Notice that $\tq[\hat{0}_{I} \,:\,\tau]=  \tq[\tau]$, with $\hat{0}_{I}$ being the partition with only one block (the set $I$), and
\[\tq[\tau\,:\,\hat{1}_{I}]=\displaystyle \bigotimes_{S\in \tau} \tq[S],\]
where $\hat{1}_{I}$ is the partition into singletons. The NPL-operad structure on $\wE \circ \tq$ presented in this work is then determined by a map
\[\displaystyle \bigoplus_{\tau \le \pi \le \rho} \tq[\tau : \pi]\otimes \tq(\rho) \longrightarrow \displaystyle \bigoplus_{H \vdash I} \tq(H),\]
where $\tq(H):=\tq[H\,:\,\hat{1}_{I}]$ for every set partition $H$ of $I$.

\

This article is organized as follows: the first part (Section \ref{sec:1}) is the introduction, the second part (Section \ref{sec:background}), divided into six paragraphs, gathers the basic notions needed in this work. Paragraphs \ref{PreliminariesPartitions} and \ref{Species} recall some basic notions on partitions and species. Paragraphs \ref{CauchyAndTwisted} and \ref{CompAndOperads}, borrowed from \cite{AM2010}, are devoted to the notions of twisted algebra and operad in species, via the Cauchy product and the substitution operation in species, respectively. We recall the definitions of the (classical) operads \textsf{Com} and \textsf{As}. In particular, we  prove that $\textsf{Com}_+$ is the unique operad structure on the positive species $\wE_+$, modulo isomorphism (Lemma \ref{ComUniq}).  We introduce in Paragraph \ref{FreeComTwAlg} the free commutative monoid $\textsf{S}(\tq):=\wE\circ \tq$ on a positive species $\tq$. Finally, Paragraph \ref{ModuleOperad} introduces the notion of modules in species, with respect to the substitution operation.\\

In Section \ref{sec:OperadsAndAlgebras}, we introduce the notion of an operad in $\tq$ compatible with a (nonunital) commutative twisted algebra. The main result here is given in Theorem \eqref{TheoremMComp}, where we endow the free commutative twisted algebra $\wE \circ \tq$ with a (nonunital) operad structure. We recover a classical operad structure on the species $\tPi$ of partitions, considered by M\'endez and Yang in \cite{MY}. We also unravel a simple, but apparently unknown, relation between the positive operad $\textsf{As}_+$ and the shuffle product (see Lemma \eqref{AsAndShuffle}). This allows us to define a new operadic operation on the species of linear set partitions $\overrightarrow{\tPi}$. \\

The last two sections are devoted to the study of NPL-operads. In Section \ref{sec:NPL}, we introduce the definition of NPL-operads in Paragraph \ref{pre-lie-like} by relaxing the nested associativity axiom down to a \textsl{nested pre-Lie axiom}. This definition is new up to our knowledge. We exhibit in Paragraph \ref{operadTonpl} the NPL-operad structure in $\wE \circ \tq$ (Theorem \ref{thm:main}), by giving an explicit expression for the partial compositions
\[\rhd_a:\wE \circ \tq[I]\otimes\wE \circ \tq[J]\to\wE \circ \tq[I\sqcup_a J].\]
The nested associativity axiom is replaced by the nested pre-Lie axiom
\[\alpha\rhd_a(\beta\rhd_b\gamma)-(\alpha\rhd_a\beta)\rhd_b\gamma)
=\beta\rhd_b(\alpha\rhd_a\gamma)-(\beta\rhd_b\alpha)\rhd_a\gamma.\]
Paragraph \ref{GlobalComp} presents an explicit form of the global multiplication $\gamma_{\wE\circ \tq}$. In Paragraph \ref{nplTonpl}, we show that the partial operation used in Theorem \ref{thm:main} is also valid when we replace the operadic structure on $\tq$ to a NPL-operadic structure. \\

Section \ref{sect:pol} gives a description of algebras over a NPL-operad. We describe a NPL-operad structure on the species $\mathcal E(V)$ of polynomial maps, where $\mathcal E(V)[I]$ is the vector space of polynomial maps from $V^{\otimes I}$ to $V$, and define an algebra over a NPL-operad $\mathcal P$ as a vector space $V$ together with a morphism of NPL-operads from $\mathcal P$ to $\mathcal E(V)$.\\

Many examples are presented in order to describe our constructions. In particular, we obtain a new operadic structure on the species $\overrightarrow{\tPi}$ of linear set partitions (Subsection \ref{LinearPartition}), and introduce NPL-operad structures on the species $\wE$ (Example \ref{NPLforE}), $\tPi$ (Example \ref{NPLforPi}), $\overrightarrow{\tPi}$ (Example \ref{NPLforArrowPi}), the species $\tc$ of cycles and the species $\mathfrak{S}$ of permutations (Subsection \ref{NPLforCandS}). 
Such species play a fundamental role in the study of the braid arrangement (see \cite{AMHyperpl}). In this setting, linear orders correspond to chambers of the arrangement, set partitions correspond to flats, and linear set partitions correspond to top-lunes. We hope that our description of operads and NPL-operads on $\wE\circ \tq$ may have geometrical applications from this point of view. For instance, the operad \textsf{As} is a particular case of a substitution operation on chambers in a general hyperplane arrangement (see Sections 4.8 and 6.5.10 in \cite{AMHyperpl}).

\

\noindent \textbf{Acknowledgements:} We thank B\'er\'enice Delcroix-Oger for illuminating discussions on distributive laws. D. M. and I. R. are partially supported by Agence Nationale de la Recherche, projet Combinatoire Alg\'ebrique, Renormalisation, Probabilit\'es libres et Op\'erades ANR-20-CE40-0007. I. R. also acknowledges support from the university of Gabes (Bourse d'Alternance). Y. V. is partially supported by the Austrian Science Fund (FWF), grant I 5788.

\section{\bf{Background on algebraic structures on Joyal's species}}\label{sec:background}
We mainly follow the presentations of \cite{J1986} and \cite{AM2010} except that we consider contravariant species. This purely conventional choice yields a \textsl{right} action of the permutation group of a finite set $I$ on components of the corresponding arity, thus matching a common practice in the operadic literature.

\subsection{Preliminaries on partitions}\label{PreliminariesPartitions}

Let $I$ be a finite set. A \emph{partition} $\pi = \{B_1, B_2, \hdots, B_k\}$ of $I$ is a set of non-empty subsets of $I$, called \emph{blocks} of the partition $\pi$, such that $I= \bigcup_{B \in \pi}B$ and $B_i \cap B_j =\emptyset$, for all $i \neq j$. The \emph{length} $\ell(\pi):=k$ of the partition $\pi$ corresponds to the number of blocks of $\pi$. We will write $\pi \vdash I$ to indicate that $\pi$ is a partition of $I$.

\medbreak

The set of partitions of $I$ is a poset, via the \emph{refinement order} on set partitions: if $\pi, \rho \vdash I$, we write $\pi \leq \rho$ if each block of $\pi$ can be obtained by merging some of the blocks of $\rho$. This poseet is a lattice, with $\hat{0}_I:=\{I\}$ and $\hat{1}_I:=\{\{i\}: i \in I\}$.

\medbreak

We write $I= S \sqcup T$ whenever there is an ordered pair $(S,T)$ of (possibly empty) sets $S$ and $T$ such that $S \cup T = I$. Such pair is called a \emph{decomposition} of $I$.

\subsection{Species}\label{Species}
Let us consider the following categories:
\begin{itemize}
\item $\mathsf{Fin}$, whose objects are finite sets and whose morphisms are bijections between finite sets;
\item $\mathsf{Vect}_{\mathbb K}$, often abbreviated $\mathsf{Vect}$, whose objects are vector spaces over some given ground field $\mathbb K$ and whose morphisms are linear maps between vector spaces.
\end{itemize}
    
Both $\mathsf{Fin}$ and $\mathsf{Vect}$ are symmetric monoidal categories, with the cartesian and tensor products, respectively.

\medbreak 
A vector species, or simply a species, is a functor
\[\mathsf{Fin}^{\text{op}} \longrightarrow \mathsf{Vect}.\]

A morphism between species $\tp$ and $\tq$ is a natural transformation between the functors $\tp$ and $\tq$. A species $\tp$ is said to be \emph{connected} if $\tp[\emptyset] = \mathbb{K}$, and \emph{positive} if $\tp[\emptyset]= (0)$. In particular, the \emph{positive part} of a species $\tp$ is the new species $\tp_+$ given by
\[\tp_+[I]:= 
\begin{cases}
\tp[I],& \text{ if } I \neq \emptyset;\\
(0), & \text{ if } I = \emptyset.
\end{cases}\]

\subsection{Cauchy product and twisted algebras}\label{CauchyAndTwisted}

The \emph{Cauchy product} of two vector species $\tp$, $\tq$ is given by
\begin{equation}
(\tp \cdot \tq)[I]:= \displaystyle \bigoplus_{I= S \sqcup T}\tp[S] \otimes  \tq[T],
\end{equation}
for any finite set $I$. If $I = S \sqcup T$ and $I \xrightarrow{\sigma}J$ is a bijection between finite sets, consider the map $\tp[\sigma|_S]\otimes \tq[\sigma|_T]: \tp[S]\otimes \tq[T] \to \tp[\sigma(S)]\otimes \tq[\sigma(T)]$, where $\sigma|_S$ and $\sigma|_T$ are the restrictions of the bijection $\sigma$ to $S$ and $T$, respectively. The map $(\tp\cdot \tq)[\sigma]$ is then defined as
\[(\tp\cdot \tq)[\sigma]:=\bigoplus_{I = S\sqcup T}\tp[\sigma|_S]\otimes \tq[\sigma|_T].\]

The species $\mathbf{1}$, defined as $\mathbf{1}[\emptyset]:= \mathbb{K}$ and $\mathbf{1}[I]:=(0)$ if $I$ is non-mepty, is the unit element for the Cauchy product. This turns the category $\Ss$ of species into a symmetric monoidal category, where the braiding $\beta: \tp \cdot \tq \to \tq \cdot \tp$ is the natural transformation where each $\beta_I$ is the direct sum of the linear maps
\[\tp[S]\otimes \tq[T] \to \tq[T] \otimes \tp[S],\]
for each decomposition $I=S \sqcup T$.

\

A \emph{twisted algebra} in species is a monoidal object in the monoidal category $(\Ss,\cdot, \mathbf{1})$. Explicitly, a twisted algebra $(\taa, \mu, \iota)$ consists of a species $\taa$ with morphisms of species 
\[\mu : \taa \cdot \taa \to \taa \qquad \text{and} \qquad \iota: \mathbf 1 \to \taa,\]
subject to associativity and unitality conditions. For a finite set $I$ and a decomposition $I= S \sqcup T$, we have linear maps
$$\mu_{S,T} : \taa[S] \otimes \taa[T] \to \taa[I] \qquad \text{and} \qquad \iota_{\emptyset}: \mathbb{K} \to \taa[\emptyset],$$
satisfying the following properties.  	

\medbreak

\begin{itemize}
\item Naturality: for each decomposition $I=S \sqcup T$ of a finite set $I$, and for each bijection $\sigma: I \to J$ we must have
\begin{displaymath}
\xymatrix@C=5cm@R=2cm{
\taa[S] \otimes \taa[T]\; \ar[r]^{\mu_{S,T}} \ar[d]_{\taa[\sigma|_S] \otimes \taa[\sigma|_T]} & \taa[I]\ \ar[d]^{\taa[\sigma]}\\
\taa[\sigma(S)] \otimes \taa[\sigma(T)]  \ar[r]_{\mu_{\sigma(S),\sigma(T)}} & \taa[J]}
\end{displaymath}

\
  	
\item Associativity: for each decomposition $I= R \sqcup S \sqcup T$,
\begin{displaymath}
\xymatrix@C=4cm@R=2cm{
\taa[R] \otimes \taa[S] \otimes \taa[T]\; \ar[r]^{\text{id}_{\taa[R]} \otimes\mu_{S,T}} \ar[d]_{\mu_{R , S}\otimes \text{id}_{\taa[T]}} & \taa[R] \otimes \taa[S\sqcup T]\ \ar[d]^{\mu_{R, S\sqcup T}}\\
\taa[R \sqcup S] \otimes \taa[T] \ar[r]_{\mu_{R \sqcup S,T}} & \taa[R \sqcup S \sqcup T]}
\end{displaymath}

\

\item Unitality: for every finite set $I$, the following diagrams commute:
\[\xymatrix{
\taa[I]\ar@{=}[drr]&& \taa[\emptyset]\otimes \taa[I]\ar[ll]_{\mu_{\emptyset,I}}\\
&&\mathbb{K}\otimes \taa[I]\ar[u]_{\iota_\emptyset\otimes \smop{id}_{\taa[I]} }}\hskip 16mm
\xymatrix{
\taa[I]\otimes \taa[\emptyset]\ar[rr]^{\mu_{I,\emptyset}}&&\taa[I]\\
\taa[I]\otimes\mathbb K\ar[u]^{\smop{id}_{\taa[I]}\otimes \iota_\emptyset}\ar@{=}[rru]}\]
\end{itemize}
\

If $x \in \taa[S]$ and $y \in \taa[T]$, we write $x \cdot y:=\mu_{S,T}(x \otimes y) \in \taa[S \sqcup T]$.

\

A twisted algebra $(\taa, \mu)$ is \emph{commutative} if the following diagram commutes:
\[\xymatrix{
\taa[S] \otimes \taa[T] \ar[dr]_{\mu_{S,T}} \ar[rr]^\beta && \taa[T]\otimes \taa[S] \ar[dl]^{\mu_{T,S}}\\
&\taa[I] &\\}
\]
for every finite set $I$ and for every decomposition $I = S \sqcup T$.

\

\begin{example}\label{TwAlgE}
The \emph{exponential species} $\wE$ is defined by $\wE[I]:= \mathbb{K}\{*_I\}$ for all finite set $I$. We denote by $*_{I}$ the element $1 \in \mathbb{K} = \wE[I]$. For every finite set $I$, the direct sum of the maps
\begin{align*}
\mu_{S,T}:\wE[S]\otimes \wE[T]&\to \wE[I]\\ *_{S}\otimes *_{T}&\mapsto *_{S \cup T},
\end{align*}
for every $I = S\sqcup T$, define a twisted algebra on $\wE$. The unit is given by 
\begin{align*}
\iota_{\emptyset}:\mathbb{K}&\to \wE[\emptyset]\\ 1_{\mathbb{K}}&\mapsto *_{\emptyset}.
\end{align*}

As $S \cup T = I = T \cup S$ for every decomposition $I = S \sqcup T$, $(\wE, \mu, \iota)$ is a commutative twisted algebra.
\end{example}

\

\begin{example}\label{TwAlgL}
A \emph{linear order} on a finite set $I$ is a bijection $\ell: [n] \to I$, where $n:=|I|$. If $\sigma: I \to J$ is a bijection between finite sets, the map $\ell \mapsto \sigma \circ \ell$ sends a linear order $\ell$ of $I$ to a new linear order on $J$. This defines the species $\wL$ of linear orders. Formally,
\[\wL[I]:= \mathbb{K}\{\ell: [n] \xrightarrow{\cong} I\},\]
whenever $n = |I|$. There is a unique linear order on the empty set, denoted by $\varnothing$. This makes $\wL$ a connected species.

\medbreak

If $\ell$ is a linear order on $I$, we write $i_1 \leq_\ell i_2$ if $\ell^{-1}(i_1)\leq \ell^{-1}(i_2)$, for all $i,j \in I$. This endows the set $I$ with a total order, for each linear order $\ell$. In general, we write $\ell = i_1|i_2| \cdots |i_n$ to denote the linear order $\ell$ of $I$, for which $I=\{i_1, i_2, \hdots, i_n\}$ and $i_1 <_\ell i_2 <_\ell \cdots <_\ell i_n$.

\medbreak

If $S \subseteq I$ and $\ell$ is a linear order on $I$, there exists a unique linear order on $S$, denoted by $\ell \cap S$, such that $s_1 \leq_{\ell \cap S} s_2$ if and only if $s_1 \leq_\ell s_2$. Writing $\ell$ as a word, $\ell \cap S$ is the new word obtained by removing elements in $I \setminus S$ and keeping the elements from $S$  in the same relative order.

\medbreak

The species $\wL$ is endowed with a commutative twisted algebra structure, given by the notion of \emph{shuffling}. If $I = S\sqcup T$, let $\ell_1$ and $\ell_2$ linear orders of $S$ and $T$, respectively. A shuffle of $\ell_1$ and $\ell_2$ is a linear order $\ell$ of $I$ such that $\ell \cap S = \ell_1$ and $\ell \cap T = \ell_2$. We denote by $\mathsf{Sh}(\ell_1, \ell_2)$ the set of all shuffles of $\ell_1$ and $\ell_2$. From here, we define a map $\sh: \wL \cdot \wL \to \wL$ with $I$-component given by the direct sums of the linear maps
\[\sh_{S,T}: \wL[S]\otimes \wL[T] \to \wL[I]\]
\[\ell_1 \otimes \ell_2 \mapsto \sum_{\ell \in \mathsf{Sh}(\ell_1, \ell_2)}\ell,\]
with $I = S \sqcup T$. The unit is given by $\iota_{\emptyset}(1_{\mathbb{K}}):=\varnothing$.
\end{example}

\begin{remark}\rm
The species of linear order $\wL$ is also endowed with a natural structure of non-commutative twisted algebra, given by the concatenation $\ell|\ell'$ of linear orders. Moreover, the concatenation is compatible with the coproduct dual to the shuffle operation (``deshuffling'' coproduct), making $\wL$ a twisted bialgebra. In this way, it is usual to consider the shuffle operation on the dual species $\wL^*$ (this is the approach taken in \cite{AM2010}, see example 8.24). The map sending a linear order on $I$ to the sum of all linear orders on $I$ is a map of twisted algebras, from $\wL$ (endowed with the concatenation operation) to $\wL^*$ (endowed with the shuffle operation). As it is not necessary in this work to make passages from a species to its dual, we consider the shuffle operation directly on $\wL$.
\end{remark}

\

\subsection{Composition and symmetric operads}\label{CompAndOperads}
Let $\tp, \tq$ be a species, with $\tq$ positive. The \emph{composition} operation $\tp\circ \tq$ is defined as follows: for any finite set $I$,
\[(\tp \circ \tq)[I]:= \displaystyle \bigoplus_{\pi \vdash I}\tp[\pi] \otimes \tq(\pi),\]
where for a given partition $\pi$ of $I$, $\tq(\pi) :=\displaystyle \bigotimes_{B \in \pi} \tq[B]$. Here, the symmetry of the tensor product of vector spaces guarantees the well-definiteness of $\tq(\pi)$, also called the \emph{unordered product}. Elements of this vector space are commutative monomials written as $\bigotimes_{B \in \pi}\alpha_B$, where $\alpha_B \in \tq[B]$ and the tensor product is taken over the partition $\pi$. See Example 1.30 in \cite{AM2010} for an explicit description of unbracketed and unordered tensor products of vector spaces using universal properties.

\medbreak 
\begin{notation}\label{notation-alpha}
For later use, let us introduce the following notation: for any element
\[\alpha=\bigotimes_{S\in X}\alpha_S\in \tq(\pi)\]
and for any subset $Z\subset \pi$ we set
\[\alpha_Z:=\bigotimes_{S\in Z}\alpha_S \quad \text{ and } \quad \alpha_{\widehat Z}:=\alpha_{\pi\setminus Z}.\]
Also, for any $B,C\in X$, we will also use the notation
\[\alpha_{\widehat{B}}:=\alpha_{\widehat{\{B\}}} \quad \text{ and } \quad  \alpha_{\widehat B,\widehat C}:=\alpha_{\widehat{\{B,C\}}}.\]
\end{notation}

\

The unit for the composition $\circ$ is the operad $\mathbb I$ defined by $\mathbb I[I]=\mathbb K\{\star\}$ for $I=\{\star\}$ being a singleton, and $\mathbb{I}[I]=(0)$ whenever $|I|\neq 1$. The species $\mathbb I$ is sometimes denoted by $\mathbf X$ in the literature, e.g. in \cite[Paragraph 8.1.2]{AM2010}. Let $\Ss_+$ be the category of positive species. Then, the triple $(\Ss_+, \circ, \mathbb{I})$ forms a monoidal category. It is not braided nor symmetric.

\

Let us now introduce some notations used by Aguiar-Mahajan in \cite[Paragraph B.1.1]{AM2010} which will be useful in the next section. Given a finite set $I$, consider partitions $\pi$ and $\rho$ of $I$, with $\rho$ refining $\pi$. There is a unique map $f:\rho \to \pi$ that sends each block $C$ of $\rho$ to the unique block $B$ of $\pi$ such that $C\subseteq B$. The fiber of $f$ over $B \in\pi$ is the set whose elements are the blocks of $\rho$ that refine the block $B$ of $\pi$. For any positive species $\tq$, we write
\[\tq[\pi\,:\,\rho]:= \displaystyle \bigotimes_{B\in \pi} \tq[f^{-1}(B)].\]
In particular, $\tq[\hat{0}_{I} \,:\,\pi]=  \tq[\pi]$, and
\[\tq[\pi\,:\,\hat{1}_{I}]= \tq(\pi)=\displaystyle \bigotimes_{B\in \pi} \tq[B].\]

In particular, associativity of the composition operation in species follows from the isomorphisms
\[((\tp \circ \tq)\circ \trr)[I] \cong \bigoplus_{\hat{0}_1 \leq \pi \leq \rho \leq \hat{1}_I} \tp[\hat{0}_1 : \pi] \otimes \tq[\pi : \rho] \otimes \trr[\rho: \hat{1}_I] \cong (\tp \circ (\tq \circ \trr))[I],\]

for all species $\tp, \tq, \trr$, with $\tq, \trr$ positives. 

\

\begin{notation}
For two nonempty sets $S,T$, and for $s \in S$, we write:
\[S \sqcup_{s} T:=\left( S\backslash \{s\}\right) \uplus T.\]
\end{notation}
  
\

A (symmetric) \emph{operad} is a monoid in $(\Ss_{+}, \circ)$. More explicitly, a positive symmetric operad is a positive species $\tq$ with associative and unital morphisms of species 
\[\gamma: \tq \circ \tq \to \tq \qquad \text{and} \qquad \eta: \mathbb{I} \to \tq,\]
such that for each nonempty finite set $I$ and for each partition $\pi \vdash I$, there is a linear map 
\[\gamma_\pi: \tq[\pi] \otimes \left( \displaystyle \bigotimes_{B \in \pi}\tq[B]\right)\to \tq[I]\]
called \emph{global operadic composition}. For every finite set $I$, $\gamma_I$ is the direct sum of all operadic compositions $\gamma_\pi$, for $\pi \in I$. Also, for each singleton $\{\star\}$, there is a unit map, 
\[\eta_\star : \mathbb{K} \to \tq[\{\star\}]\]
satisfying naturality, associativity and unitality axioms.

\

The correspondence between symmetric operads and monoids in the monoidal category $(\Ss_+, \circ, \mathbb{I})$) goes back to the work of Kelly in \cite{K2005}. In \cite{Ma1996b}, Markl introduces a description of the operadic operation of an operadic in terms of \emph{partial compositions}. In this language, a symmetric operad is a positive species $\tq$ such that, for any two nonempty finite sets $S,T$, and for any $s \in S$, there exist  elements $ u_{s} \in \tq[\{s\}]$ and a partial composition operation 
\[\circ_{s} : \tq[S]\otimes \tq[T] \to \tq[S \sqcup_{s}T],\] 
satisfying the followings axioms:

\begin{itemize}
 \item Associativity: for $ x \in \tq[S]$, $y \in \tq[T]$,
\begin{itemize}
\item[(A1)] $(x \circ_{s}y)\circ_{s'}z=(x \circ_{s'}z)\circ_{s}y$, if $s$ and $s'$ are two distinct elements of $S$;
\item[(A2)] $(x \circ_{s}y)\circ_{t}z=x \circ_{s}(y\circ_{t}z)$, if $s \in S$ and $t \in T$.
\end{itemize}

\medbreak
  		
\item Naturality: given two bijections $\sigma_{1}:S \to S'$ and $\sigma_{2}:T \to T'$ and $s \in S$. Then, for every $x \in \tq[S]$, $y \in \tq[T]$,

\begin{itemize}
\item[(N1)] $\tq[\sigma_{1}](x) \circ_{\sigma(s)} \tq[\sigma_{2}](y)= \tq[\sigma](x \circ_{s}y)$, where
\[\sigma:=(\sigma_{1})|_{S\setminus\{s\}} \cup \sigma_{2}:S \,{\sqcup}_{s} T \to S' \sqcup_{\sigma_{1}(s)} T';\]
\item[(N2)] if $s_{1},s_{2}\in S$, then for any bijection $\sigma:\{s_{1}\}\to \{s_{2}\}$ we have 
\[\tq[\sigma](u_{s_{1}})=u_{s_{2}}.\]
\end{itemize}
\medbreak
\item Unitality: for $s \in S$ and $x \in \tq[S]$, we have
\begin{itemize}
\item[(U1)] $u_{\star}\circ_{\star}x=x$ for any singleton $\{\star\}$;
\item[(U2)] $x \circ_{s}u_{s}=x$.
\end{itemize}
\end{itemize}

\

\begin{example}[The positive commutative operad]\label{OperadCom}
The positive commutative operad $\mathsf{Com}_+$ is an operad on $\wE_+$ defined by
\[*_S\circ_a *_T=*_{S\sqcup_a T},\] 
for any finite sets $S,T$ and for any $a\in S$. For every finite set $I$, the $I$-components of the unit $\mathbb{I}[I] \to \wE_+[I]$ are given by $1_{\mathbb{K}} \mapsto *_I$. 
The axioms of an operad are obviously verified. 
\end{example}

The above operad structure on $\wE$ is the only one, modulo isomorphism. 

\begin{lemma}\label{ComUniq}
$\mathsf{Com}_+$ is the unique operad structure on $\wE_+$.
\end{lemma}
\begin{proof}
Remark that the partial composition of any operad structure on $\wE$ are necessary written as
\[*_S \circ_s *_T=\alpha(|S|,|T|)*_{S\sqcup_s T}\]
for $S,T$ finite sets and $s\in S$. The coefficient $\alpha$ depends only on the cardinalities because of the naturality axioms. Parallel and nested associativity then respectively yield
\begin{eqnarray}
    \alpha(a,b)\alpha(a+b-1,c)&=&\alpha(a,c)\alpha(a+c-1,b)\label{parallele},\\
    \alpha(a,b)\alpha(a+b-1,c)&=&\alpha(a,b+c-1)\alpha(b,c)\label{emboite}.
\end{eqnarray}
for any positive integers $a,b,c$ (think $a=|S|$, $b=|T|$ and $c=|U|$). Unitality axioms yield
\begin{equation}\label{unite}
  \alpha(1,a)=\alpha(a,1)=1.  
\end{equation}
Putting $a=1$ in \eqref{parallele} and using \eqref{unite}, we obtain the symmetry $\alpha(b,c)=\alpha(c,b)$. Both equations \eqref{parallele} and \eqref{emboite} therefore become equivalent. Now given any map $\lambda:\mathbb N\to \mathbb K\setminus\{0\}$ with $\lambda(0)=\lambda(1)=1$, any map of the form
\[\alpha'(a,b)=h(a,b)\alpha(a,b)\]
with
\[h(a,b):=\frac{\lambda(a)\lambda(b)}{\lambda(a+b-1)}\] also verifies \eqref{parallele} and \eqref{emboite}. Here appears the second cohomology group of the monoid $\mathbb Z_{>0}$ endowed with the associative product $a*b:=a+b-1$, acting trivially on $\mathbb K\setminus\{0\}$. The $\alpha$'s are cocycles and the $h$'s are coboundaries. Two cohomologous cocycles come from two isomorphic operad structures. The semi-group itself is obviously isomorphic to $\mathbb Z_{\ge 0}$ endowed with the usual addition. The cohomology groups of $\mathbb Z_{\ge 0}$ turn out to be isomorphic to those of the group $\mathbb Z$ \cite{B1961}. It is well known that $H^2(\mathbb Z, \mathbb K^*)$ is trivial \cite[III.1 Example 1]{B1982}, which proves the claim.
\end{proof}

\

\begin{example}[The positive associative operad]\label{OperadAss}
The positive associative operad $\mathsf{As}_+$ is the operad on $\wL_+$ defined as follows: for every pair of disjoint sets $S,T$ and $s \in S$, consider the map
\[\circ_s: \wL[S]\otimes \wL[T] \to \wL[S \sqcup_sT]\]
\[\ell_1|s|\ell_2 \otimes \ell \mapsto \ell_1|\ell|\ell_2.\]

Here, the chosen linear order on $S$ is denoted by the concatenation $\ell_1|s|\ell_2$, thus showing the relative position of $s$.
\end{example}

\

\subsection{Free commutative monoid on a positive species}\label{FreeComTwAlg}

Let $\tq$ be a positive vector species. As $\wE[I] = \mathbb{K}$ for all finite set $I$, we have
\[(\wE \circ\tq)[I]= \bigoplus_{\pi \vdash I} \wE[\pi] \otimes \tq(\pi) = \bigoplus_{\pi \vdash I}\tq(\pi).\]

The species $\tq$ naturally embeds into $\wE \circ \tq$ via $\eta(\tq): \tq \to \wE \circ \tq$, whether $\eta(\tq)_I: \tq[I] \to \tq[\hat{1}_I]$ is the isomorphism induced by the natural bijection $I \cong \{\{i\}: i \in I\}=: \hat{1}_I$. Given a positive species $\tq$, there is a twisted algebra structure on $\wE \circ \tq$ as follows: for every finite set $I$ and a decomposition $I = S \sqcup T$, consider the morphism
\[\mu_{S,T}: (\wE\circ\tq)[S] \otimes (\wE\circ\tq)[T] \to (\wE \circ \tq)[I]\]
defined as the direct sum of the following maps
\[\tq(\pi) \otimes \tq(\rho )\overset{\cong}{\longrightarrow} \tq(\pi \cup \rho ),\]
as $\pi$ and $\rho$ run over all partitions of $S$ and $T$, respectively.
Then, the map $\mu_I$ is defined as the direct sum of the maps $\mu_{S,T}$, for every decomposition $I = S \sqcup T$. The unit is given by the identification of $\mathbb{K} = (\wE \circ \tq)[\emptyset]$. This defines a commutative twisted algebra structure on $\wE \circ \tq$, which in turns defines a functor $\mathsf{S}: \Ss_+ \to \mathsf{Mon}^{\mathsf{co}}(\Ss)$, sending any positive species $\tq$ to the commutative twisted algebra $\wE \circ \tq$. Moreover, this construction yields the \emph{free commutative twisted algebra} over $\tq$.

\begin{proposition}[Theorem 11.13, \cite{AM2010}]\label{PropUnivEq}
Let $\tq$ be a positive species. Consider $(\tp, \mu, \iota)$ a commutative twisted algebra and let $\zeta: \tq \to \tp_+$ be a species map. Then, there exists a unique morphisms $\hat{\zeta}: \mathsf{S}(\tq)_+ \to (\tp_+, \mu)$ of commutative twisted algebras such that the following diagram commutes:
\begin{equation}\label{UnivEq}
\xymatrix{
\mathsf{S}(\tq)_+ \ar[rr]^{\hat{\zeta}_+}&& (\tp, \mu, \iota)_+\\
&\tq \ar[lu]^{\eta(\tq)} \ar[ru]_{\zeta}&
}
\end{equation}
\end{proposition}

\subsection{Modules over an operad}\label{ModuleOperad}

Let $(\tp, \gamma, \eta)$ be an operad. A \emph{left $\tp$-module} is a species $\tm$ with a map $\chi: \tp\circ \tm \to \tm$ which is associative and unital. That is, the following diagrams commute:
\[
\xymatrix{
\tp \circ \tp \circ \tm \ar[r]^-{\text{id}_{\tp}\circ \chi} \ar[d]_{\gamma \circ \text{id}_{\tm}} &\tp \circ \tm \ar[d]^{\chi}\\
\tp \circ \tm \ar[r]_-{\gamma}&\tm
} \qquad , \qquad
\xymatrix{
\mathbb{I}\circ \tm \ar[r]^{\eta \circ \text{id}_{\tm}} \ar[d]_{=}& \tp \circ\tm \ar[dl]^-{\chi}\\
\tm
}
\]

Similarly, $\tm$ is a \emph{right $\tp$-module} if there is an associative and unital map $\chi: \tm \circ \tp \to \tm$.

\medbreak

Let $\tq$ be another operad. A \emph{$\tp$-$\tq$-bimodule} is a species $\tm$ which is a left $\tp$-module, a right $\tq$-module and such that the left and right operad actions $\chi_\ell: \tp \circ \tm \to \tm$ and $\chi_r: \tm \circ \tq \to \tm$ commute:
\[
\xymatrix{
\tp \circ \tm \circ \tq \ar[r]^-{\text{id}_{\tp}\circ \chi_r} \ar[d]_-{\chi_{\ell} \circ \text{id}_{\tq}} & \tp \circ \tm \ar[d]^{\chi_{\ell}}\\
\tm \circ \tq \ar[r]_-{\chi_{r}} & \tm
}
\]

\medbreak

In particular, any positive species $\tp$ with an operadic structure is a $\tp$-$\tp$-bimodule in the obvious way. As for operads, a left $\tp$-module structure on $\tm$ is equivalent to the existence of partial composition operations
\[\circ_s: \tp[S]\otimes \tm[T]\to \tm[S \sqcup_s T],\]
one for every pair $S,T$ of nonempty finite sets and for every element $s \in S$, satisfying associativity, naturality and unitality conditions. The same equivalent notion applies to right-modules and bi-modules. 

\

\section{Operads and compatible twisted algebras} \label{sec:OperadsAndAlgebras}
Let $\tq$ be a positive species. Define a new species $\tq^\circ$ as follows:
\[\tq^\circ[I]:=
\begin{cases}
\mathbb{K}, & \text{if } I = \emptyset;\\
\tq[I], & \text{otherwise}.
\end{cases}
\]

In this section, we consider a special type of operads $(\tq, \gamma)$, endowed with an extra commutative twisted algebra structure $(\tq^\circ, \mu)$ on $\tq^\circ$, satisfying a compatibility relation with the operad structure $\gamma$. This yields a nonunital operad structure on $\wE \circ\tq$. 

\begin{remark}\rm
Let $\Ss^\circ$ and $\Ss_+$ be the categories of connected and positive species, respectively. The assumption of a a twisted algebra $(\tq^\circ, \mu)$ can be replaced with a ``positive'' twisted algebra on $\tq$ as follows. The category $\Ss_+$ is monoidal via the operation $\odot$ given by $\tp_1 \odot \tp_2:= \tp_1 \cdot \tp_2 + \tp_1 + \tp_2$. Moreover, the pair of functors $(-)_+: (\Ss^\circ, \cdot) \to (\Ss_+, \odot)$ and $(-)^\circ: (\Ss_+, \odot)\to (\Ss^\circ, \cdot)$ are bistrong adjoint equivalence. In particular, any twisted algebra in $(\Ss^\circ, \cdot)$ is equivalent to a twisted algebra in $(\Ss_+, \odot)$ (Proposition 8.44 in \cite{AM2010}).
\end{remark}

\medbreak

\begin{notation}
Given a twisted algebra $(\tq, \mu)$, we write $x \bullet y$ for $\mu_{S,T}(x \otimes y)$, whether $I = S \sqcup T$, $x \in \tq[S]$ and $y \in \tq[T]$.
\end{notation}

\

\begin{definition}\label{OperadCompatible}
Let $(\tq, \gamma)$ be an operad, with partial compositions $\{\circ_s\}_s$. Let $S,T$ be nonempty sets, $\pi= \{B_1, B_2, \hdots, B_k\} \vdash S$ with elements $\alpha_i \in \tq[B_i]$ for $1 \leq i \leq k$, and $\beta \in \tq[T]$.

\medbreak

The operad $\gamma$ is said to be $\mu$-compatible, for a commutative twisted algebra structure $(\tq^\circ, \mu)$, if for every $s \in S$ the following relation holds
\begin{equation}\label{mcomp}
(\alpha_1 \bullet \alpha_2 \bullet \cdots \bullet \alpha_k)\circ_s \beta = (\alpha_{i_0} \circ_s \beta) \bullet \alpha_1 \bullet \alpha_2 \bullet \cdots \bullet \widehat{\alpha_{i_0}} \bullet \cdots \bullet \alpha_k,
\end{equation}
where $B_{i_0}$ is the block of $\pi$ containing $s$.
\end{definition}
\medbreak
We now discuss on the global formulation of the notion of a $\mu$-compatible operad. Given a commutative twisted algebra $(\tq^\circ, \mu)$, the map $\tq \xrightarrow{\text{id}} (\tq^\circ)_+=\tq$ induces a map $\chi_\mu:\wE\circ\tq \to \tq$ of (nonunital) commutative twisted algebras, via the universal property \eqref{UnivEq}. It is straightforward to show that $\chi_\mu$ is compatible with the (unique) operad $\textsf{Com}_+$ structure on $\wE$. Hence, the commutative twisted algebra structure on $\tq$ induces a left $\textsf{Com}_+$-module in $\tq$. 

\medbreak

Suppose now that $(\tq, \gamma_\tq)$ is a $\mu$-compatible operad. As $\tq$ is a right $\tq$-module, condition \eqref{mcomp} is equivalent to have a $\textsf{Com}_+$-\,$\tq$-bimodule on $\tq$. In particular, the following diagram commutes:
\begin{equation}
\xymatrix{
\wE \circ \tq \circ \tq \ar[r]^-{\text{id}_{\wE} \circ \gamma_\tq} \ar[d]_-{\chi_{\mu}\circ\text{id}_{\tq}} & \wE \circ \tq \ar[d]^-{\chi_\mu}\\
\tq \circ \tq \ar[r]_-{\gamma_\tq} & \tq
}
\end{equation}

\begin{lemma}\label{LemmaComp}
\phantom{a}
\begin{itemize}
    \item[$(i)$] A nonunital commutative twisted algebra structure on $\tq$ is equivalent to a {\normalfont$\textsf{Com}_+$}-module on $\tq$.
    \item[$(ii)$] A $\mu$-compatible operad $(\tq, \gamma_\tq)$ is equivalent to a {\normalfont$\textsf{Com}_+$}-\,$\tq$-bimodule structure on $\tq$. In this case, we have {\normalfont$\chi_\mu \, (\text{id}_{\wE} \circ \gamma_\tq) = \gamma_\tq \, (\chi_\mu \circ \text{id}_{\tq})$}.
\end{itemize}
\end{lemma}

We are now ready to define the operad structure on $\wE\circ\tq$.

\begin{definition}
Let $I=S\sqcup T$ and $ s \in S$. Let $\pi \vdash S$, $\rho \vdash T$ and let $s \in S$. Let $B_s$ the block in $X$ containing $s$. Consider
\[\alpha =\bigotimes_{B\in \pi}\alpha_{B} \quad \text{ and } \quad \beta = \bigotimes_{C\in \rho}\beta_{C}. \]

Let $(\tq, \gamma_\tq)$ be an operad, with partial compositions $\{\circ_s\}_{s}$.
If $(\tq, \gamma_{\tq})$ is a $\mu$-compatible operad, we define partial composition maps $\sq_s$ given by
\[\sq_{s}:(\wE\circ\tq)[S]\otimes (\wE\circ\tq)[T]\longrightarrow (\wE\circ\tq)[S \sqcup_{s}T]\] 
\begin{equation}\label{DefSq}
\alpha \,\sq_s \,\beta:= {\big(}\alpha_{B_s} \circ_s \mu(\beta){\big)}\,\alpha_{\widehat{B_s}}
\end{equation}
\end{definition}

The right-hand side is a commutative concatenation in $(\wE\circ \tq)[S \sqcup_sT]$, formed by the ``letter'' $\alpha_{B_s} \circ_s \mu(\beta)$, obtained by multiplying in the twisted algebra $(\tq, \mu)$ the elements of the monomial $\beta=\bigotimes_{C\in \rho}\beta_{C}$ and composing through $\gamma_\tq$ with $\alpha_{B_s}$, and the monomial $\alpha_{\widehat{B_s}}$ formed by all elements in $\alpha$ except $\alpha_{B_s}$.

\

The next theorem shows that the partial compositions $\{\sq_s\}_s$ endow $\wE \circ \tq$ with a nonunital operad structure. From \eqref{DefSq}, the associated global operation $\gamma: (\wE \circ \tq) \circ (\wE \circ \tq)\to \wE \circ \tq$ satisfies
\begin{equation}
\gamma = (\text{id}_\wE \circ \gamma_\tq)(\text{id}_{\wE \circ \tq} \circ \chi_\mu).
\end{equation}

\

\begin{theorem}\label{TheoremMComp}
If $(\tq, \gamma_\tq)$ is a $\mu$-compatible operad, then the partial operations $\{\sq_s\}_s$ define a nonunital operad structure $\gamma$ on $\wE \circ \tq$.
\end{theorem}

\begin{remark}\rm
It is still possible to define a right-unit for $\wE \circ \tq$ as follows. Let $(\tq, \gamma_\tq, \eta_\tq)$ be an operad. Let $S$ be a nonempty finite set. Then, by unitality of $\eta_\tq$, there exists an element $u_s \in \tq[\{s\}]$ for every $s \in S$ such that $x \circ_s u_s = x$, for all $x \in \tq[S]$. As
\[\tq[\{s\}] \cong \wE[\{\{s\}\}] \otimes \tq[\{s\}] = (\wE \circ \tq)[\{s\}],\]
we identify $u_s \in \tq[\{s\}]$ with its corresponding element in $(\wE\circ \tq)[\{s\}]$ via the above isomorphism.  If $\pi \vdash S$ and $\alpha=\bigotimes_{B \in \pi}\alpha_B$, with $s \in B_s$, then
\[\alpha \, \sq_s u_s = (\alpha_{B_s} \circ_s u_s)\alpha_{\widehat{B_s}}=\alpha,\]

so $u_s$ is a right unit for $\sq_s$. In the other hand, if $|S|\geq 2$ and $|\pi|\geq 2$, we have
\[u_\star \sq_\star \beta = u_\star \circ_\star \mu(\beta) = \mu(\beta) \neq \beta,\]
for every $\beta \in \tq(\pi)$.
\end{remark}

\begin{proof}[(Proof of the Theorem)]
Consider the following diagram.

\begin{equation}
\SelectTips{eu}{12}
\xymatrix@C=.8cm@R=1.7cm{
&&\ar @{} [dd] |{\text{(I)}} \wE \circ \tq \circ \wE \circ \tq \ar[dr]^-{\text{id}_\wE \, \circ \, \text{id}_\tq \,\circ \,\chi_\mu}&& \\
&\ar @{} [dd] |{\text{(II)}} \wE \circ \tq \circ \tq \circ \wE \circ \tq \ar[dr]^-{\quad \text{id}_\wE \, \circ \, \text{id}_\tq \, \circ \, \text{id}_\tq \, \circ \, \chi_\mu} \ar[ru]^-{\text{id}_\wE \, \circ \, \gamma_\tq \, \circ \, \text{id}_\wE \, \circ \, \text{id}_\tq \qquad }&&\wE \circ \tq \circ \tq \ar[dr]^-{\text{id}_\wE \, \circ \, \gamma_\tq} \ar @{} [dd] |{\text{(III)}}& \\
\wE \circ \tq \circ \wE \circ \tq \circ \wE \circ \tq \ar[dr]_{\text{id}_\wE \, \circ \, \text{id}_\tq \, \circ \, \text{id}_\wE \, \circ \, \text{id}_\tq \, \circ \, \chi_\mu \qquad} \ar[ur]^-{\text{id}_\wE \, \circ \, \text{id}_\tq \, \circ \, \chi_\mu \, \circ \, \text{id}_\wE \, \circ \, \text{id}_\tq \quad \qquad }&&\ar @{} [dd] |{\text{(IV)}}\wE \circ \tq \circ \tq \circ \tq \ar[ur]_-{\text{id}_\wE \, \circ \, \gamma_\tq \, \circ \, \text{id}_\tq} \ar[dr]_-{\text{id}_\wE \, \circ \, \text{id}_\tq \, \circ \, \gamma_\tq}&&\wE \circ \tq \\
&\wE \circ \tq \circ \wE \circ \tq \circ \tq \ar[ru]^-{\text{id}_\wE \, \circ \, \text{id}_\tq \, \circ \, \chi_\mu \, \circ \, \text{id}_\tq \qquad} \ar[dr]_-{\text{id}_\wE \, \circ \, \text{id}_\tq \, \circ \, \text{id}_\wE \, \circ \, \gamma_\tq \qquad}&&\wE \circ \tq \circ \tq \ar[ru]_{\text{id}_\wE \, \circ \, \gamma_\tq}& \\
&&\wE \circ \tq \circ \wE \circ \tq \ar[ru]_-{\quad \text{id}_\wE \, \circ \, \text{id}_\tq \, \circ \, \chi_\mu}&& 
}
\end{equation}

\

Squares (I) and (II) trivially commute. Commutativity of diagram (III) follows from the $\tq$-$\tq$-bimodule structure on $\tq$. Finally, since $(\tq, \gamma_\tq)$ is a $\mu$-compatible operad, item (ii) in Lemma \eqref{LemmaComp} guarantees the commutativity of square (IV). Therefore, the above diagram commutes.

\medbreak

Now, the composite of the sequence of arrows from $\wE \circ \tq \circ \wE \circ \tq \circ \wE \circ \tq$ to $\wE \circ \tq$ on the top and on the bottom are, respectively, $\gamma\,(\gamma \circ \text{id}_{\wE \circ \tq})$ and $\gamma \, (\text{id}_{\wE \circ \tq} \circ \gamma)$. As the diagram commute both maps are equal, implying the associativity of $\gamma$.
\end{proof}

\medbreak

\begin{remark}\rm
Given an operad $(\tq, \gamma_\tq)$, a similar argument can be used to endow the species $\wL \circ \tq$ with a nonunital operadic structure, analogue to Theorem \eqref{TheoremMComp}, starting from a $\wL$-$\tq$-bimodule structure on $\tq$. The latter is equivalent to the following noncommutative analogue of the relation \eqref{mcomp}:
\begin{equation}
(\alpha_1 \bullet \alpha_2 \bullet \cdots \bullet \alpha_{i_0} \bullet \cdots \bullet \alpha_k)\circ_s \beta:=\alpha_1 \bullet \alpha_2 \bullet \cdots \bullet (\alpha_{i_0}\circ_s \beta) \bullet \cdots \bullet \alpha_k,
\end{equation}
whenever $I=S \sqcup T$, $\pi=\{B_1, B_2, \hdots, B_k\} \vdash S$, with $s \in B_{i_0}$ and $\alpha_i \in \tq[B_i]$, $\beta \in \tq[T]$. Here, $\bullet$ denotes a noncommutative twisted algebra map $\mu: \tq \cdot \tq \to \tq$. In this case, each partial composition maps $\sq_s$ 
\[\sq_{s}:(\wL\circ\tq)[S]\otimes (\wL\circ\tq)[T]\longrightarrow (\wL\circ\tq)[S \sqcup_{s}T]\] 
is defined as the sum of the components
\[\tq(F) \otimes \tq(G)\to \tq(H)\]
\begin{equation}
(\alpha_1 \otimes \alpha_2 \cdots \otimes \alpha_{i_0} \otimes \cdots \otimes \alpha_k) \,\sq_s \,\beta:= \alpha_1 \otimes \alpha_2 \otimes \cdots \otimes {\big(}\alpha_{i_0} \circ_s \mu(\beta){\big)} \otimes \cdots \otimes \alpha_k,
\end{equation}
where $F = B_1 \sqcup B_2 \sqcup \cdots \sqcup B_k$, $G$ and  $H:= B_1 \sqcup \cdots \sqcup (B_{i_0} \sqcup_s T) \sqcup \cdots \sqcup B_k$ are decompositions of $S$, $T$ and $S \sqcup_s T$, respectively, and $\alpha_1 \otimes \cdots \otimes \alpha_k \in \tq(F)$, $\beta \in \tq(G)$.

\end{remark}

\

\subsection{Operad on set partitions}\label{OperadonPi}

Consider the species $\tPi:= \wE \circ \wE_+$ of set partitions. Formally,
\[\tPi[I]:= \mathbb{K}\{\pi : \pi\vdash I\},\]
for every finite set $I$.

\medbreak 

Let $S,T$ be two non-empty disjoint sets, and let $s \in S$. Consider $\pi \vdash S$ a partition of $S$, where $B_s$ denotes the block of $\pi$ containing $s$. Recall that $\wE$ is endowed with a twisted algebra structure $\mu$ (see Example \eqref{TwAlgE}), and the operad structure $\textsf{Com}_+$ (see Example \eqref{OperadCom}). The identity on sets
\[S\sqcup_s T = (B_s \sqcup_s T) \cup (S \setminus B_S)\]
implies the relation \eqref{mcomp}. More precisely, if $S,T$ are nonempty set, $\pi=\{B_1, \hdots, B_k\}\vdash S$ with $B_{i_0}$ being the block containing $s$ and $\alpha = \bigotimes_{B \in \pi}*_B \in \wE_+(\pi)$, $\beta = *_T \in \wE_+[T]$, then
\begin{align*}
(*_{B_1} \bullet \cdots \bullet *_{B_k}) \circ_s *_T = *_{B_1 \cup \cdots \cup B_k} \circ_s *_T &= *_{S \sqcup_s T}\\
&= *_{(B_s \sqcup_s T) \cup (S \setminus B_S)} = *_{B_s \sqcup_s T} \, \bullet \, (*_{B_1}\bullet \cdots \bullet \widehat{*_{B_{i_0}}} \bullet \cdots \bullet *_{B_k}).
\end{align*}

Therefore, we have the following result.
\begin{lemma}
The operad {\normalfont$\textsf{Com}_+$} is $\mu$-compatible.
\end{lemma}

By Theorem \eqref{TheoremMComp}, $\tPi$ is endowed with a nonunital operad structure, which we describe in the following lines. First, if $\pi = \{B_1, B_2, \hdots, B_k\} \vdash S$, we make use of the following identification between pure tensors in $\wE_+(\pi)$ and partitions of $S$:
\[\wE_+(\pi) \ni *_{B_1}*_{B_2}\cdots *_{B_k} \longleftrightarrow \{B_1, B_2, \hdots, B_k\} = \pi \in \tPi[S].\]

As a partial operation, the operad structure on $\tPi$ is given by
\[\pi \, \sq_s  \, \rho = (*_{B_s} \circ_s \mu(\rho))\, (\pi \setminus B_s)= (*_{B_S} \circ_s *_T)\, (\pi \setminus B_s) = (\pi \setminus B_s) \cup T,\]
where $\pi \vdash S$ and $\rho \vdash T$. 

\

This is equivalent to the following formulation. Given a finite set $I$ and partitions $\pi \vdash I$, $\tau \vdash \pi$, define the \emph{coinduced partition} $\text{ind}_\pi(\tau) \vdash I$ as 
\[\text{ind}_\pi(\tau) := \left\{\bigcup_{B \in C}B: C \in \tau \right\}.\]
Hence, the operadic composition $\tPi \circ \tPi \to \tPi$ has components
\[\tPi[\pi] \otimes \tPi(\pi) \to \tPi[I]\]
\[\tau \otimes \bigotimes_{B \in \pi}\rho_{B} \mapsto \text{ind}_\pi(\tau), \]
where $\rho_B \vdash B$, for every block $B$ of $\pi$. This (set-theoretical) operad on $\tPi$ is considered by M\'endez and Yang in \cite{MY} (see Example 3.12).

\

\subsection{Operad on linear set partitions}\label{LinearPartition}

A \emph{linear set partition} of a finite set $I$ is a partition with a linear order structure on each block. We write $\pi \Vdash I$ whether $\pi$ is a linear set partition of $I$. A linear set partition on $I$ can be also interpreted as a partial order on $I$, where each connected component is a chain.

\

The species of linear set partitions is denoted by $\overrightarrow{\tPi}$. A direct consequence of the definition of linear set partition yields the equality in species
$\overrightarrow{\tPi} = \wE \circ \wL_+$. Using the commutative twisted algebra structure on $\wL$ given by the shuffle operation $\sh$ (see Example \eqref{TwAlgL}), together with the operad $\textsf{As}_+$, we obtain a nonunital operadic structure on  $\overrightarrow{\tPi}$.

\begin{lemma}\label{AsAndShuffle}
The operad {\normalfont $\textsf{As}_+$} is $\sh$-compatible.
\end{lemma}

\begin{proof}
Let $S,T$ be disjoint nonempty sets, with $s \in S$ fixed. If $\pi = \{B_1, \hdots, B_k\} \vdash S$, with $s \in B_{i_0}$, let $\ell_i$ be a linear order of $B_i$, for every $1 \leq i \leq k$. For the linear order $\ell_{i_0}$ containing the element $s$, we write $\ell_{i_0} = \ell_{i_0}^1|s|\ell_{i_0}^2$, where $\ell_{i_0}^1|\ell_{i_0}^2$ is a linear order of $S\setminus \{s\}$. Also, let $\ell$ be a linear order on $T$.

\

We need to prove
\[(\ell_1 \,\sh \, \cdots \,\sh \,\ell_k)\circ_s \ell = (\ell_{i_0} \circ_s \ell) \,\sh \,\ell_1 \,\sh \,\cdots \, \sh\, \widehat{\ell_{i_0}} \,\sh \,\cdots \,\sh\, \ell_k,\]
where $\circ_s$ corresponds to the partial operation associated to $\textsf{As}_+$. As $B_1, \hdots, B_k,T$ are disjoint sets, both expressions on the left and right-hand side above consist of sum of monomials each with coefficient 1.  Let $M_1$ and $M_2$ be the set of monomials of the left and right-hand side, respectively. Then, it is straightforward to show that both sets $M_1$ and $M_2$ consist of all linear orders $\ell'$ of $S \sqcup_s T$ such that
\[\ell' \cap B_i = \ell_i, \, \forall \, i \in [k]\setminus \{i_0\} \quad , \quad \ell' \cap (B_{i_0}\setminus \{i_0\})= \ell_{i_0}^1|\ell_{i_0}^2 \quad \text{ and } \quad \ell'\cap T = \ell.\]
Therefore, $M_1 = M_2$.
\end{proof}

\

We now describe the global composition coming from Theorem \eqref{TheoremMComp}. Given a finite set $I$, let $\pi = \{B_1, \hdots, B_k\} \vdash I$, $\tau = \{\ell^\tau_1, \hdots, \ell^\tau_r\} \Vdash \pi$ and linear partitions $\rho_B \Vdash B$, one for every block $B \in \pi$. Each $\ell^\tau_i$ is a linear order on a subset of blocks of the partition $\tau$. For every $B \in \pi$, we denote by $\textsf{Sh}(\rho_B)$ the set of linear orders of $B$ resulting from the shuffle of all possible linear orders on the blocks of $\rho_B$. 

\begin{example}
Let $I=[9]$, $\pi = {\big\{} \{1,2, 3,4\}, \{5,6\}, \{7,8,9\} {\big\}}$, $\tau= {\big\{}\{1,2,3,4\} \, ; \, \{7,8,9\}| \{5,6\}{\big\}}$ and 
\[\rho_1 = {\big\{} 3|4|2;1 {\big\}}, \rho_2 = {\big\{} 6|5 {\big\}}, \rho_3={\big\{} 7|9 ;8 {\big\}}.\]
We have:
\[\textsf{Sh}(\rho_1)={\big\{}3|4|2|1, 3|4|1|2, 3|1|4|2, 1|3|4|2{\big\}} \, , \, \textsf{Sh}(\rho_2)={\big\{}6|5{\big\}} \, \text{ and } \,\textsf{Sh}(\rho_3)= {\big\{}7|9|8, 7|8|9, 8|7|9{\big\}}.\]
\end{example}

\

We continue with the description of the global composition. Fix a collection of linear orders $L= \{\ell_1, \hdots, \ell_k\}$, with $\ell_i \in \textsf{Sh}(\rho_{B_i})$, for $1 \leq i \leq k$. Given a linear order $\ell^\tau= B_{i_1}| \cdots | B_{i_k}$ in $\tau$, let $\ell^\tau(L):= \ell_{B_{i_1}}|\cdots| \ell_{B_{i_k}}$ be the linear order of $\bigcup_{B \in \ell^\tau}B$, obtained by choosing the fixed linear orders in $L$ that correspond to the blocks $B_{i_1}, \hdots,  B_{i_k}$ forming the linear $\ell^\tau$, and concatenating them according to $\ell^\tau$.

\

Consider now the \emph{coinduced linear partition} $\overrightarrow{\text{ind}}_\pi(\tau;L) \Vdash I$, defined as
\begin{equation}
\overrightarrow{\text{ind}}_\pi(\tau;L):= {\big\{}\ell^\tau(L): \ell^\tau \in \tau{\big\}}.
\end{equation}

\

From here, the operadic map $\overrightarrow{\tPi} \circ \overrightarrow{\tPi} \to \overrightarrow{\tPi}$ has components
\[\overrightarrow{\tPi}[\pi] \otimes \overrightarrow{\tPi}(\pi) \to \overrightarrow{\tPi}[I]\]
\begin{equation}
\tau \otimes \bigotimes_{B \in \pi}\rho_B \mapsto \sum_{\ell_1 \in \textsf{Sh}(\rho_{B_1}), \hdots, \ell_k \in \textsf{Sh}(\rho_{B_k})} \overrightarrow{\text{ind}}_\pi(\tau;\{\ell_1, \hdots, \ell_k\}).
\end{equation}

\

\begin{example}
Continuing with the preceding example, let $\ell_1 = 3|4|2|1 \in \textsf{Sh}(\rho_1), \ell_2 = 6|5 \in \textsf{Sh}(\rho_2)$ and $\ell_3= 7|9|8 \in \textsf{Sh}(\rho_3)$. The coinduced linear partition $\overrightarrow{\text{ind}}_\pi(\tau;\{\ell_1, \ell_2, \ell_3\})$ is
\[\overrightarrow{\text{ind}}_\pi(\tau;\{\ell_1, \ell_2, \ell_3\}) = {\big\{} 3|4|2|1, 7|9|8|6|5 {\big\}}.\]

The image of $\tau \otimes \bigotimes_{B \in \pi}\rho_B$ under the resulting operad map $\gamma$ is a sum of 12 terms:
\begin{align*}
\gamma\left(\tau \otimes \bigotimes_{B \in \pi}\rho_B \right)=\,& {\big\{} 3|4|2|1, 7|9|8|6|5 {\big\}} + {\big\{} 3|4|1|2, 7|9|8|6|5 {\big\}} + {\big\{} 3|1|4|2, 7|9|8|6|5 {\big\}} + {\big\{} 1|3|4|2, 7|9|8|6|5 {\big\}} \\
& {\big\{} 3|4|2|1, 7|8|9|6|5 {\big\}} + {\big\{} 3|4|1|2, 7|8|9|6|5 {\big\}} + {\big\{} 3|1|4|2, 7|8|9|6|5 {\big\}} + {\big\{} 1|3|4|2, 7|8|9|6|5 {\big\}}\\
&\\
&{\big\{} 3|4|2|1, 8|7|9|6|5 {\big\}} + {\big\{} 3|4|1|2, 8|7|9|6|5 {\big\}} + {\big\{} 3|1|4|2, 8|7|9|6|5 {\big\}} + {\big\{} 1|3|4|2, 8|7|9|6|5 {\big\}}.
\end{align*}
\end{example}

\

\section{\bf{Nested Pre-Lie operads}} \label{sec:NPL}

\subsection{NPL-operads}\label{pre-lie-like}
We keep the notations of Section \ref{sec:background}.

\begin{definition}
A \textbf{NPL-operad} is a species endowed with the axioms of an operad, except that the unitality axioms are discarded and the nested associativity axiom is replaced by the weaker \textbf{nested pre-Lie axiom} (NPL):
\smallbreak
\begin{itemize}
\item{\rm (NPL)} \quad $(x \circ_{s}y)\circ_{t}z-x \circ_{s}(y\circ_{t}z)=(y \circ_{t}x)\circ_{s}z-y \circ_{t}(x\circ_{s}z)$,
\noindent if $s \in S$ and $t \in T$.\end{itemize}
\end{definition}

Any operad in the classical sense is a NPL-operad, as Axiom (A2) obviously implies (NPL). 

\

Let $\tp$ be a species such that $\tp[I]=(0)$ if $|I|\neq 1$. Then we say that $\tp$ is \emph{concentrated in degree one}. In this situation, an operad structure on $\tp$ is equivalent to an associative algebra structure on $\tp[I]$. Analogously, an NPL-operad structure on $\tp$ is equivalent to a (unitary) pre-Lie structure on $\tp[I]$.

\[\scalebox{0.3}{\preoperad}\]
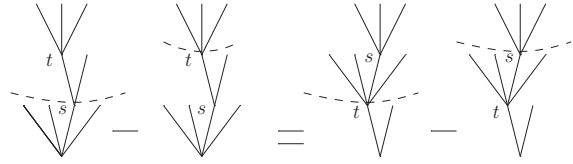
\centerline{\small The nested pre-Lie axiom NPL}\\

\

Recall that the derivative $\tp'$ of a species $\tp$ is defined by $\tp'[I]:=\tp[I\sqcup\{*\}]$ for any finite set $I$. The ``star'' element $*$ stands for any extra ghost element. It is well known (see e.g. \cite[Paragraph 3.4.3]{M2015}) that for any operad $\tp$, the derivative $\tp'$ is endowed with a canonical monoid structure, obtained from the partial composition at the ghost element. Similarly, the derivative of any NPL-operad admits a canonical pre-Lie algebra structure in the monoidal category of species endowed with the Cauchy product (a twisted pre-Lie algebra structure in the sense of \cite{J1986}).

We shall see in the next paragraph that a whole family of examples of NPL-operads is given by $\wE\circ\tq$ as soon $\tq$ is an operad.

\

\begin{example}\label{NPLforE}
The simplest example of NPL-operad is given by $\wE$. Indeed, the species $\mathbb I$, unit for the composition product $\circ$, is itself an operad. The only nonzero (partial) compositions are given by $\star \circ_\star \ast=\ast$ for any pair $(\{\star\},\{\ast\})$ of singletons. Applying \eqref{square-partial} to $\tq=\mathbb I$, we see that the NPL-operad structure $\rhd$ on the species $\wE\circ\mathbb I=\wE$ is given by
\[*_S\,\rhd_s *_T=|T|*_S\circ_s *_T=|T|*_{S\sqcup_s T}\]
for any finite sets $S,T$ and for any $s\in S$. The nested pre-Lie property is illustrated by
\[*_S\rhd_s (*_T\rhd_t *_U)-(*_S\rhd_s *_T)\rhd_t *_U= *_T\rhd_t (*_S\rhd_s *_U)-(*_T\rhd_t *_S)\rhd_s *_U=|U|(|U|-1)*_{S\sqcup_sT\sqcup_tU}.\]

Note that $*_{\{t\}}$ is a right unit, but not a left unit.
\end{example}

\

\subsection{NPL-operad structure on $\wE\circ \tq$ from an operad on $\tq$}\label{operadTonpl}
Let $(\tq,\gamma_\tq)$ be an operad,  with partial composition maps $\{\circ_s\}_s$. Given $I=S\sqcup T$ and a fixed element $ s \in S$, we define a partial composition \[\rhd_{s}:(\wE\circ\tq)[S]\otimes (\wE\circ\tq)[T]\longrightarrow (\wE\circ \tq)[S \sqcup_{s}T]\] 
on $\wE\circ\tq$ as 
\begin{equation}\label{square-partial}
\alpha\rhd_s\beta:=\sum_{C\in \tau}  (\alpha_{B_s}\circ_s\beta_C)\,\alpha_{\widehat{B_s}}\,\beta_{\widehat C},      
\end{equation}
where we omit the unordered tensor product sign $\otimes$ in the right-hand side.
   
\    
    
\begin{theorem}\label{thm:main}
If $(\tq,\gamma_\tq)$ is an operad, then the partial operations $\{\vartriangleright_s\}_s$ define a NPL-operad structure on $\wE\circ\tq$.
\end{theorem}
    
\begin{proof}
\phantom{a}
\begin{itemize}
\item  Associativity: Let $I=S\sqcup T \sqcup R$, and fix $s,s' \in S$ and $t \in T$. Let $\pi \vdash S$,  $\tau \vdash T$ and $\rho \vdash R$. Consider $\alpha \in \tq(\pi)$, $\beta \in \tq(\tau)$ and $\gamma  \in \tq(\rho)$. Denote by $B_{s}$ (resp $B_{s'}$) the block of $\pi$ that contains ${s}$ (resp ${s'}$) and $B_{t}$ the block of $\tau$ that contains ${t}$. We have two possible cases, $B_{s} = B_{s'}$ or $B_{s} \ne B_{s'}$.\\
    
- Case $B_{s} = B_{s'}$: we have
     \begin{eqnarray*}
     (\alpha\rhd_s\beta)\rhd_{s'}\gamma
     &=&\sum_{C\in \tau}\Big((\alpha_{B_s}\circ_s\beta_C)\alpha_{\widehat{B_s}}\beta_{\widehat C}\Big)\rhd_{s'}\gamma\\  
     &=&\sum_{C\in \tau,\,D\in \rho}\big((\alpha_{B_s}\circ_s\beta_C)\circ_{s'}\gamma_D\big)\alpha_{\widehat{B_s}}\beta_{\widehat C}\gamma_{\widehat D}\\
     &=&\sum_{D\in \rho,\,C\in \tau}\big((\alpha_{B_s}\circ_{s'}\gamma_D)\circ_{s}\beta_C\big)\alpha_{\widehat{B_s}}\gamma_{\widehat D}\beta_{\widehat C} \quad (\text{by parallel associativity for }\tq)\\
     &=&(\alpha\rhd_{s'}\gamma)\rhd_{s'}\beta.
     \end{eqnarray*}
    - Case $B_{s} \ne B_{s'}$: we compute
\begin{eqnarray*}
     (\alpha\rhd_s\beta)\rhd_{s'}\gamma
     &=&\sum_{C\in \tau}\Big((\alpha_{B_s}\circ_s\beta_C)\alpha_{\widehat{B_s}}\beta_{\widehat C}\Big)\rhd_{s'}\gamma\\  
     &=&\sum_{C\in \tau,\, D\in \rho}(\alpha_{B_s}\circ_s\beta_C)(\alpha_{B_{s'}}\circ_{s'}\gamma_D)\beta_{\widehat C}\gamma_{\widehat D}\\
     &=&\sum_{D\in \rho,\, C\in \tau}(\alpha_{B_{s'}}\circ_{s'}\gamma_D)(\alpha_{B_s}\circ_s\beta_C)\gamma_{\widehat D}\beta_{\widehat C}\\
     &=&(\alpha\rhd_{s'}\gamma)\rhd_{s}\beta,
\end{eqnarray*}
hence (A1) is proved. Let us now prove (NPL). We have:
\begin{eqnarray*}
(\alpha\rhd_s\beta)\rhd_t\gamma
&=&\sum_{C\in \tau}\big((\alpha_{B_s}\circ_s\beta_C)\alpha_{\widehat{B_s}}\beta_{\widehat C}\big)\rhd_t\gamma\\
&=&\sum_{C\in \tau\setminus\{C_t\}, \,D\in \rho}
(\alpha_{B_s}\circ_s \beta_C)\alpha_{\widehat{B_s}}(\beta_{C_t}\circ\gamma_D)\beta_{\widehat C,\widehat{C_t}}\gamma_{\widehat D}\\
&& +\sum_{D\in \rho}\big((\alpha_{B_s}\circ_s\beta_{C_t})\circ_t\gamma_D\big)\alpha_{\widehat{B_s}}\beta_{\widehat C_t}\gamma_{\widehat D}.
\end{eqnarray*}
     \noindent On the other hand,
\begin{eqnarray*}
\alpha\rhd_s(\beta\rhd_t\gamma)
&=&\sum_{D\in \rho}\alpha\rhd_s\big((\beta_{C_t}\circ_t\gamma_D)\beta_{\widehat{C_t}}\gamma_{\widehat D}\big)\\
&=&\sum_{D,D'\in \rho,\, D\neq D'}(\alpha_{B_s}\circ_s\gamma_{D'})(\beta_{C_t}\circ_t\gamma_D)\alpha_{\widehat{B_s}}\beta_{\widehat{C_t}}\gamma_{\widehat D,\widehat{D'}}\\
&&+\sum_{C\in \tau\setminus\{C_t\}, \,D\in \rho}
(\alpha_{B_s}\circ_s \beta_C)\alpha_{\widehat{B_s}}(\beta_{C_t}\circ\gamma_D)\beta_{\widehat C,\widehat{C_t}}\gamma_{\widehat D}\\
&& +\sum_{D\in \rho}\big((\alpha_{B_s}\circ_s\beta_{C_t})\circ_t\gamma_D\big)\alpha_{\widehat{B_s}}\beta_{\widehat C_t}\gamma_{\widehat D}.
\end{eqnarray*}
     \noindent Subtracting the result of the first computation from the result of the second one, we therefore get 
\begin{eqnarray*}
 \alpha\rhd_s(\beta\rhd_t\gamma)-  (\alpha\rhd_s\beta)\rhd_t\gamma)
 &=&\sum_{D,D'\in \rho,\, D\neq D'}(\alpha_{B_s}\circ_s\gamma_{D'})(\beta_{C_t}\circ_t\gamma_D)\alpha_{\widehat{B_s}}\beta_{\widehat{C_t}}\gamma_{\widehat D,\widehat{D'}}\\
 &=&\sum_{D,D'\in \rho,\, D\neq D'}(\beta_{C_t}\circ_t\gamma_D)(\alpha_{B_s}\circ_s\gamma_{D'})\beta_{\widehat{C_t}}\alpha_{\widehat{B_s}}\gamma_{\widehat D,\widehat{D'}}\\
 &=&\beta\rhd_t(\alpha\rhd_s\gamma)-  (\beta\rhd_t\alpha)\rhd_s\gamma).
\end{eqnarray*}
    \item Naturality: let $I = S \sqcup T$ be a partition. Fix an element $s \in S$. Consider $\sigma_{1} : S \to S'$ and $\sigma_{1} : T \to T'$ two bijections and let
    $\sigma := (\sigma_{1})|_{S\setminus\{s\}} \bigcup \sigma_{2}$. Let $\pi \vdash S$ and $\tau \vdash T$, and let $B_{s}$ be the block of $\pi$ containing $s$.
    \medbreak
    If $ \alpha \in \tq(\pi)$ and $\beta \in \tq(\tau)$, we have
\begin{eqnarray*}
(\wE\circ\tq)[\sigma](\alpha\rhd_s\beta)
&=&(\wE\circ\tq)[\sigma]\left(\sum_{C\in \tau}(\alpha_{B_s}\circ_s\beta_C)\alpha_{\widehat{B_s}}\beta_{\widehat C}\right)\\
&=&\sum_{C\in \tau}\tq[\sigma](\alpha_{B_s}\circ_s\beta_C)\tq[\sigma](\alpha_{\widehat{B_s}})\tq[\sigma](\beta_{\widehat C})\\
&=&\tq[\sigma_1](\alpha_{B_s})\circ_{\sigma(s)}\tq[\sigma_2](\beta_C)\tq[\sigma_1](\alpha_{\widehat{B_s}})\tq[\sigma_2](\beta_{\widehat C})\\
&=&(\wE\circ\tq)[\sigma_1](\alpha)\rhd_{\sigma(s)}(\wE\circ\tq)[\sigma_2](\beta),
\end{eqnarray*}
       which proves $(N_{1})$. For $(N_{2})$, consider $s_{1}, s_{2} \in S$ and $\sigma : {s_{1}} \to {s_{2}}$ the function given by $\sigma(s_{1}) = s_{2}$. If $u_{s_{1}} \in (\wE \circ \tq)[{s_{1}}] = \tq[{s_{1}}]$ and $u_{s_{2}}\in
   (\wE \circ \tq)[{s_{2}}] = \tq[{s_{2}}]$, we have
   $$(\wE \circ \tq)[\sigma](u_{s_{1}}) = \tq[\sigma](u_{s_{1}}) = u_{s_{2}},$$
   where the second equality is given by the naturality property $(N_{2})$ of $\circ_s$.
   \end{itemize}
\end{proof}


\

\begin{remark}\rm
From the assumptions of the previous theorem, since $(\wE \circ \tq)[\{t\}] = \tq[\{t\}]$, the unitality property (U2) for $\rhd$ come directly from the ones for $\circ$. As for example \ref{NPLforE}, axiom (U1) does not hold.
\end{remark}

\medbreak

\begin{remark}\rm
From the proof of condition (NPL) in Theorem \ref{thm:main}, we remark that the difference 
$\alpha\rhd_s(\beta\rhd_t\gamma)-(\alpha\rhd_s)\beta)\rhd_t\gamma$ vanishes whenever the partition $\rho$ has a single block $D$, i.e. if $\gamma\in \tq[D]$. This reflects the fact that $\wE\circ\tq$ is a right module on the operad $\tq$, the action
\[\chi:(\wE\circ\tq)\circ\tq\to\wE\circ\tq\]
being given by $\mop{id}_{\wE}\circ \gamma_{\tq}$.
\end{remark}

\

\begin{example}\label{NPLforPi}
The operad $\textsf{Com}_+$ endows the species of partitions $\tPi = \wE \circ \wE_+$ with a NPL-operadic structure, via Theorem \eqref{thm:main}. Given $S,T$ nonempty finite sets, $s\in S$, let $\pi \vdash S$ and $\tau \vdash T$. Let $B_s \in \pi$ be the block of $\pi$ containing the element $s$. The NPL-operad on $\tPi$ is given by
\[\rhd_s: \tPi[S] \otimes \tPi[T] \to \tPi[S \sqcup_s T]\]
\[\pi \rhd_s \tau:= \sum_{C \in \tau} \{B_s \circ_s C\}\cup (\pi \setminus B_s)\, (\tau \setminus C),\]
where $\circ_s$ corresponds to the partial composition of $\textsf{Com}_+$.

\end{example}

\medbreak

\begin{example}\label{NPLforArrowPi}
Recall the species of linear set partitions $\overrightarrow{\tPi}=\wE\circ \wL_+$ introduced in \eqref{LinearPartition}. The positive species $\wL_+$ is endowed with the operad $\textsf{As}_+$ (see Example \eqref{OperadAss}), with partial compositions $\{\circ_s\}_s$. Hence, using Theorem \eqref{thm:main}, there is a NPL-operad structure on $\overrightarrow{\tPi}$. Given $S,T$ nonempty finite sets, $s \in S$, let $\pi=\{\ell^\pi_1, \hdots, \ell^\pi_m\} \Vdash S$, $\tau=\{\ell^\tau_1, \hdots, \ell^\tau_n\} \Vdash T$ be linear set partitions of $S$ and $T$, respectively. Let $\ell_s \in \pi$ be the linear order in $\pi$ containing the element $s$. Hence, by Theorem \eqref{thm:main}, the partial compositions
\[\overrightarrow{\tPi}[S] \otimes \overrightarrow{\tPi}[T] \to \overrightarrow{\tPi}[S \sqcup_s T]\]
\[\pi \rhd_s\tau:= \sum_{\ell^\tau \in \tau} \{\ell_s \circ_s \ell^\tau\} \cup (\pi\setminus\{\ell_s\})\cup (\tau \setminus  \{\ell^\tau\}) 
\]
define a NPL-operad structure on $\overrightarrow{\tPi}$.

\end{example}

\

\subsection{The global composition} \label{GlobalComp}
We now want to describe the NPL-operad map
\begin{equation}
\gamma:(\wE \circ \tq) \circ (\wE \circ \tq) \to (\wE \circ \tq)
\end{equation}
from Theorem \eqref{thm:main}, through the iteration of partial compositions, given an operad structure $\gamma_\tq$ on $\tq$. 

\medbreak
Let $I$ be a finite set. We have:
\begin{align*}
(\wE \circ \tq) \circ (\wE \circ \tq) [I] &= \displaystyle \bigoplus_{\pi \vdash I}\, (\wE \circ \tq) [\pi] \otimes (\wE \circ \tq) (\pi) \\
&= \displaystyle \bigoplus_{\pi \vdash I}\, \Biggl(\displaystyle \bigotimes_{\tau \vdash \pi} \tq(\tau) \Biggr) \otimes \Biggl(\displaystyle \bigotimes_{\rho_{B} \vdash B \in \pi} \tq(\rho_{B})\Biggr)\\
&= \bigoplus_{\tau \le \pi \le \rho} \tq[\tau : \pi]\otimes \tq(\rho).
\end{align*}
We show now that the $I$-th component of the map $\gamma$ is of the form
\begin{equation}
\gamma_I: \bigoplus_{\tau \le \pi \le \rho} \tq[\tau : \pi]\otimes \tq(\rho) \longrightarrow  \bigoplus_{H \vdash I} \tq(H).    
\end{equation}
In particular, for three fixed partitions $\tau$, $\pi$, $\rho$ of $I$ such that $\tau \le \pi \le \rho$  the global map $\gamma$ send $\tq[\tau : \pi]\otimes \tq(\rho)$ to $\displaystyle \bigoplus_{H \vdash I} \tq(H)$. We will restrict to partitions $H$ such that $\tau \le H \le \rho$, more precisely to partitions $H_{\mathrm{S}}$ defined as follows.\\
    
Let $g_{\pi} : \rho  \twoheadrightarrow \pi$ be the canonical surjection that sends each block of $\rho$ for the single block of $\pi$ that contains it.
For each section $\mathrm{S}: \pi \to \rho$ of $g_\pi$, we construct the partition $H_{\mathrm{S}}$ of $\rho$ by assembling all the blocks of $\mathrm{S}(\pi)$, which are in the same block $B$ of $\tau$, into a single block $D_{B}^{S}$, leaving the other blocks of $\rho$ unchanged. We denote by $\mathrm{S}^{-1}$ the inverse of the bijection $\mathrm{S}: \pi \to \mathrm{S}(\pi) \subseteq\rho$, and by $\mathrm{S}^{-1}_{B}$ the inverse of $\mathrm{S}_{|B}: \pi_{|B} \to \mathrm{S}(\pi_{|B}) \subseteq \rho_{|B}$. The union  $\displaystyle \bigcup_{T \in \pi}\mathrm{S}(T)$ therefore coincides with $\bigcup_{B\in \tau}D^{\mathrm{S}}_{B}.$

\

\noindent A typical element in $\tq[\tau:\pi] \otimes \tq(\rho)$ , to which we apply $\mu_{\wE \circ \tq}$, is written $\alpha\otimes \beta$ with
$$\displaystyle \alpha=\bigotimes_{B \in \tau} \alpha_{B} \hbox{ and } \beta=\bigotimes_{C \in \rho}\beta_{C},$$ with $\alpha_{B} \in \tq[\pi_{|B}]$ and $\beta_{C}\in \tq[C]$. We have, using Notation \ref{notation-alpha}, the following result.

\medbreak

\begin{proposition}
The global NPL-operadic map $\gamma:(\wE \circ \tq) \circ (\wE \circ \tq) \to (\wE \circ \tq)$ has components
\[\bigoplus_{\tau \le \pi \le \rho} \tq[\tau : \pi]\otimes \tq(\rho) \longrightarrow  \bigoplus_{H \vdash I} \tq(H)\]
\begin{equation}\label{FormulaGamma}
\alpha \otimes \beta \mapsto \sum_{\substack{S: \pi \,\hookrightarrow \, \rho \\ g_{\pi}\circ S \,= \, \text{id}_{\pi} }}\gamma_{\tq}\left(\alpha_B^S\,\otimes \,\beta_{S\left(\pi\srestr{B}\right)}\right)\,\beta_{\widehat{S\left(\pi\srestr{B}\right)}}.
\end{equation}
Here, $\alpha_{B}^{\mathrm{S}}:= \tq[\mathrm{S}_{B}^{-1}](\alpha_{B}) \in \tq[\mathrm{S}(\pi_{\vert B})]$ and $\gamma_{\tq}\left(\alpha_B^S\otimes \beta_{S\left(\pi\srestr{B}\right)}\right)\in \tq[D^{\mathrm{S}}_B].$

\end{proposition}
\vskip 8mm
\tikzset{every picture/.style={line width=0.75pt}} 

\begin{tikzpicture}[x=0.75pt,y=0.75pt,yscale=-1,xscale=1]

\draw    (350,10) -- (350,493) ;
\draw    (1,251) -- (700,251) ;
\draw [line width=2.25]    (26,24) -- (49,51) ;
\draw [line width=2.25]    (49,51) -- (49,17) ;
\draw [line width=2.25]    (49,51) -- (72,24) ;
\draw [line width=2.25]    (126,24) -- (149,51) ;
\draw [line width=2.25]    (149,51) -- (149,17) ;
\draw [line width=2.25]    (149,51) -- (172,24) ;
\draw [line width=2.25]    (100,51) -- (100,17) ;
\draw [line width=2.25]    (77,125) -- (100,152) ;
\draw [line width=2.25]    (100,152) -- (123,125) ;
\draw    (49,51) -- (77,125) ;
\draw    (99,51) -- (77,125) ;
\draw    (149,51) -- (123,125) ;
\draw [line width=2.25]    (217,24) -- (240,51) ;
\draw [line width=2.25]    (240,51) -- (263,24) ;
\draw [line width=2.25]    (290,51) -- (290,17) ;
\draw [line width=2.25]    (264,125) -- (264,152) ;
\draw    (240,51) -- (264,125) ;
\draw    (290,51) -- (264,125) ;
\draw [line width=2.25]    (377,24) -- (400,51) ;
\draw [line width=2.25]    (400,51) -- (400,17) ;
\draw [line width=2.25]    (400,51) -- (423,24) ;
\draw [line width=2.25]    (477,24) -- (500,51) ;
\draw [line width=2.25]    (500,51) -- (500,17) ;
\draw [line width=2.25]    (500,51) -- (523,24) ;
\draw [line width=2.25]    (450,51) -- (450,17) ;
\draw [line width=2.25]    (424,125) -- (451,152) ;
\draw [line width=2.25]    (451,152) -- (477,125) ;
\draw    (400,51) -- (424,125) ;
\draw  [dash pattern={on 4.5pt off 4.5pt}]  (450,51) -- (424,125) ;
\draw    (500,51) -- (477,125) ;
\draw [line width=2.25]    (568,24) -- (591,51) ;
\draw [line width=2.25]    (591,51) -- (614,24) ;
\draw [line width=2.25]    (641,51) -- (641,17) ;
\draw [line width=2.25]    (615,125) -- (615,152) ;
\draw  [dash pattern={on 4.5pt off 4.5pt}]  (591,51) -- (615,125) ;
\draw    (641,51) -- (615,125) ;
\draw [line width=2.25]    (77,375) -- (100,402) ;
\draw [line width=2.25]    (100,402) -- (123,375) ;
\draw    (49,301.5) -- (77,375) ;
\draw    (99,299.5) -- (77,375) ;
\draw    (149,301.5) -- (123,375) ;
\draw [line width=2.25]    (264,375) -- (264,402) ;
\draw    (240,300.5) -- (264,375) ;
\draw    (290,298.5) -- (264,375) ;
\draw [line width=2.25]    (400.5,160) -- (423.5,186) ;
\draw [line width=2.25]    (423.5,186) -- (423.5,153) ;
\draw [line width=2.25]    (423.5,186) -- (446.5,160) ;
\draw [line width=2.25]    (453.5,160) -- (476.5,186) ;
\draw [line width=2.25]    (476.5,186) -- (476.5,153) ;
\draw [line width=2.25]    (476.5,186) -- (499.5,160) ;
\draw [line width=2.25]    (424,185) -- (451,212) ;
\draw [line width=2.25]    (451,212) -- (477,185) ;
\draw [line width=2.25]    (520,186) -- (520,153) ;
\draw [line width=2.25]    (520,186) -- (520,219) ;
\draw [line width=4]    (520,184) -- (520,188) ;
\draw [line width=2.25]    (580,186) -- (580,219) ;
\draw [line width=2.25]    (624,185) -- (651,212) ;
\draw [line width=2.25]    (651,212) -- (677,185) ;
\draw [line width=2.25]    (426,377) -- (452,402) ;
\draw [line width=2.25]    (452,402) -- (478,377) ;
\draw    (401,300) -- (426,377) ;
\draw  [dash pattern={on 4.5pt off 4.5pt}]  (451,300) -- (426,377) ;
\draw    (501,300) -- (478,377.5) ;
\draw [line width=2.25]    (616,368.5) -- (616,393.5) ;
\draw  [dash pattern={on 4.5pt off 4.5pt}]  (592,300.5) -- (616,368.5) ;
\draw    (642,298.5) -- (616,368.5) ;
\draw [line width=2.25]    (426,440) -- (452,465) ;
\draw [line width=2.25]    (452,465) -- (478,440) ;
\draw [line width=2.25]    (536,440) -- (536,465) ;
\draw (10,30) node [anchor=north west][inner sep=0.75pt]    {$\beta $};
\draw (11,135) node [anchor=north west][inner sep=0.75pt]    {$\alpha $};
\draw (314,17) node [anchor=north west][inner sep=0.75pt]    {$I$};
\draw (309,43) node [anchor=north west][inner sep=0.75pt]    {$\rho $};
\draw (304,114) node [anchor=north west][inner sep=0.75pt]    {$\pi $};
\draw (304,142) node [anchor=north west][inner sep=0.75pt]    {$\tau $};
\draw (361,30) node [anchor=north west][inner sep=0.75pt]    {$\beta $};
\draw (361,135) node [anchor=north west][inner sep=0.75pt]    {$\alpha $};
\draw (665,17) node [anchor=north west][inner sep=0.75pt]    {$I$};
\draw (660,43) node [anchor=north west][inner sep=0.75pt]    {$\rho $};
\draw (655,114) node [anchor=north west][inner sep=0.75pt]    {$\pi $};
\draw (655,142) node [anchor=north west][inner sep=0.75pt]    {$\tau $};
\draw (87,189) node [anchor=north west][inner sep=0.75pt]   [align=left] {An element of $\wE\circ \tq\circ \wE\circ \tq$};
\draw (440,228) node [anchor=north west][inner sep=0.75pt]   [align=left] {Composition with a choice of section};
\draw (10,278.4) node [anchor=north west][inner sep=0.75pt]    {$\beta $};
\draw (11,383.4) node [anchor=north west][inner sep=0.75pt]    {$\alpha $};
\draw (309,290) node [anchor=north west][inner sep=0.75pt]    {$\rho $};
\draw (304,360) node [anchor=north west][inner sep=0.75pt]    {$\pi $};
\draw (304,390) node [anchor=north west][inner sep=0.75pt]    {$\tau $};
\draw (21,439) node [anchor=north west][inner sep=0.75pt]   [align=left] {The same element with a simplified notation};
\draw (44,277) node [anchor=north west][inner sep=0.75pt]   [align=left] {1};
\draw (95,277) node [anchor=north west][inner sep=0.75pt]   [align=left] {2};
\draw (144,277) node [anchor=north west][inner sep=0.75pt]   [align=left] {3};
\draw (236,277) node [anchor=north west][inner sep=0.75pt]   [align=left] {4};
\draw (283,277) node [anchor=north west][inner sep=0.75pt]   [align=left] {5};
\draw (362,278.4) node [anchor=north west][inner sep=0.75pt]    {$\beta $};
\draw (363,383.4) node [anchor=north west][inner sep=0.75pt]    {$\alpha $};
\draw (661,290) node [anchor=north west][inner sep=0.75pt]    {$\rho $};
\draw (656,360) node [anchor=north west][inner sep=0.75pt]    {$\pi $};
\draw (656,390) node [anchor=north west][inner sep=0.75pt]    {$\tau $};
\draw (389,479) node [anchor=north west][inner sep=0.75pt]   [align=left] {Composition with the same choice of section};
\draw (396,277) node [anchor=north west][inner sep=0.75pt]   [align=left] {1};
\draw (447,277) node [anchor=north west][inner sep=0.75pt]   [align=left] {2};
\draw (496,277) node [anchor=north west][inner sep=0.75pt]   [align=left] {3};
\draw (588,277) node [anchor=north west][inner sep=0.75pt]   [align=left] {4};
\draw (635,277) node [anchor=north west][inner sep=0.75pt]   [align=left] {5};
\draw (418,425) node [anchor=north west][inner sep=0.75pt]   [align=left] {1};
\draw (471,425) node [anchor=north west][inner sep=0.75pt]   [align=left] {3};
\draw (530,425) node [anchor=north west][inner sep=0.75pt]   [align=left] {5};
\draw (570,455) node [anchor=north west][inner sep=0.75pt]   [align=left] {2};
\draw (610,455) node [anchor=north west][inner sep=0.75pt]   [align=left] {4};
\end{tikzpicture}

\begin{remark}\rm
The map $\gamma:(\wE \circ \tq) \circ (\wE \circ \tq) \to (\wE \circ \tq)$ can also be obtained from the free commutative twisted algebra on $\tq$ as follows. From an operad $(\tq, \gamma_\tq)$, consider the free left $\tq$-module on $\wE \circ \tq$, the species $\tq \circ \wE \circ \tq$. From here, let $\gamma_0:\tq \circ \wE \circ \tq \to \wE \circ \tq$ be the species map given by the rule \eqref{FormulaGamma} restricted to the elements of $\tq \circ \wE \circ \tq$. By the universal property (Theorem \eqref{PropUnivEq}) on the free twisted algebra on $\tq \circ \wE \circ \tq$, there is a unique map (of commutative twisted algebras) from $\wE \circ (\tq \circ \wE \circ \tq)$ to $\wE \circ \tq$, which is precisely $\gamma$.
\end{remark}

\medbreak

\subsection{NPL-operad structure on $\wE\circ \tq$ from a NPL-operad $\tq$}\label{nplTonpl}

Given species $\tp$ and $\tp_{\text{c}}$ related by the identity $\tp = \wE \circ \tp_\text{c}$, we say that $\tp_{\text{c}}$ is the \emph{species of connected $\tp$-structures}. By definition, $\tp_{\text{c}}[\emptyset]=(0)$.

\medbreak

The next result shows that the partial operation defined in \eqref{square-partial} allows to construct a NPL-operad structure on $\tp=\wE\circ \tp_{\text{c}}$ from a NPL-operad on its connected $\tp$-structure $\tp_{\text{c}}$.

\begin{prop}\label{PropnplTpnpl}
If $\tq$ is a NPL-operad with partial operations $\{\circ_s\}_s$, then the partial operations $\{\vartriangleright_s\}_s$ define a NPL-operad structure on $\wE\circ\tq$.
\end{prop}
\begin{proof}
The proof of (A1), (N1) and (N2) are the same as for the case when $(\tq,\gamma_\tq)$ is an operad. It therefore remains to demonstrate (NPL). We have:
\begin{eqnarray*}
\left(\alpha_{B}  \rhd_{s} \beta \right) \rhd_{t} \gamma 
&=& \displaystyle \sum_{C \in \tau}\left(  ( \alpha_{B_{s}} \circ_{s} \beta_{C})\alpha_{\widehat{B_s}}\beta_{\widehat C}\Big) \right) \, \rhd_{t} \gamma\\
&=& \sum_{C\in \tau\setminus\{C_t\}, \,D\in \rho}
(\alpha_{B_s}\circ_s \beta_C)\alpha_{\widehat{B_s}}(\beta_{C_t}\circ\gamma_D)\beta_{\widehat C,\widehat{C_t}}\gamma_{\widehat D}\\
&& +\sum_{D\in \rho}\big((\alpha_{B_s}\circ_s\beta_{C_t})\circ_t\gamma_D\big)\alpha_{\widehat{B_s}}\beta_{\widehat C_t}\gamma_{\widehat D}.
     \end{eqnarray*}
     \noindent On the other hand,
     \begin{eqnarray*}
         \alpha\rhd_s(\beta\rhd_t\gamma)
&=&\sum_{D\in \rho}\alpha\rhd_s\big((\beta_{C_t}\circ_t\gamma_D)\beta_{\widehat{C_t}}\gamma_{\widehat D}\big)\\
&=&\sum_{C\in \tau\setminus\{C_t\}, \,D\in \rho}
(\alpha_{B_s}\circ_s \beta_C)\alpha_{\widehat{B_s}}(\beta_{C_t}\circ\gamma_D)\beta_{\widehat C,\widehat{C_t}}\gamma_{\widehat D}
\\
&&+\sum_{D\in \rho}\big(\alpha_{B_s}\circ_s(\beta_{C_t}\circ_t\gamma_D)\big)\alpha_{\widehat{B_s}}\beta_{\widehat C_t}\gamma_{\widehat D}\\
&& +\sum_{D,D'\in \rho,\, D\neq D'}(\alpha_{B_s}\circ_s\gamma_{D'})(\beta_{C_t}\circ_t\gamma_D)\alpha_{\widehat{B_s}}\beta_{\widehat{C_t}}\gamma_{\widehat D,\widehat{D'}}.
\end{eqnarray*}

\noindent Subtracting the result of the first computation from the result of the second one, we therefore get 
\begin{eqnarray*}
\alpha\rhd_s(\beta\rhd_t\gamma)-\left(\alpha \rhd_{s} \beta \right) \rhd_{t} \gamma
&=&\sum_{D\in \rho}\Big[\big(\alpha_{B_s}\circ_s(\beta_{C_t}\circ_t\gamma_D)\big)-\big((\alpha_{B_s}\circ_s\beta_{C_t})\circ_t\gamma_D\big)\Big]\alpha_{\widehat{B_s}}\beta_{\widehat C_t}\gamma_{\widehat D}\\
&& +\sum_{D,D'\in \rho,\, D\neq D'}(\alpha_{B_s}\circ_s\gamma_{D'})(\beta_{C_t}\circ_t\gamma_D)\alpha_{\widehat{B_s}}\beta_{\widehat{C_t}}\gamma_{\widehat D,\widehat{D'}}\\
&=&\sum_{D\in \rho}\Big[\big(\beta_{C_t}\circ_t(\alpha_{B_s}\circ_s\gamma_D)\big)-\big((\beta_{C_t}\circ_t\alpha_{B_s})\circ_s\gamma_D\big)\Big]\alpha_{\widehat{B_s}}\beta_{\widehat C_t}\gamma_{\widehat D}\\
&& +\sum_{D,D'\in \rho,\, D\neq D'}(\alpha_{B_s}\circ_s\gamma_{D'})(\beta_{C_t}\circ_t\gamma_D)\alpha_{\widehat{B_s}}\beta_{\widehat{C_t}}\gamma_{\widehat D,\widehat{D'}}.
\end{eqnarray*}

Then 
\begin{eqnarray*}
\alpha \rhd_s(\beta \rhd_t\gamma)-\left(\alpha  \rhd_{s} \beta \right) \rhd_{t} \gamma
&=&\beta\rhd_t(\alpha\rhd_s\gamma)-\left(\beta  \rhd_{t} \alpha  \right) \rhd_{s}\gamma,
\end{eqnarray*}

which ends up the proof.
\end{proof}

\

\subsection{A NPL-operad on connected structures}\label{NPLforCandS}
So far, we have explored the structure of NPL-operads on species of the form $\wE \circ \tq$, starting from an operad or a NPL-operad on $\tq$.  This leads to the question whether there exists a NPL-operad on a species of connected structures of certain kind. In the following, we present a non-trivial NPL-operad structure on the \emph{species of cycles}, which corresponds to the connected components of the \emph{species of permutations}.

\

Let $I$ be a finite set. A \emph{cycle} on $I$ is a bijection $c: I \xrightarrow[]{\cong} I$ such that, for every $i \in I$, $\min\{m: m \geq 1,  c^m(i) = i\}=|I|$. Equivalently, the action on $I$ of the subgroup generated by $c$ is free and transitive. If $c$ is a cycle of $I$, we write $c \circlearrowleft I$. Let
\[\tc[I]:=\mathbb{K}\{c: I \to I \, | \, c \circlearrowleft I\}.\]

Any bijection $I \xrightarrow{\sigma}J$ between finite sets defines a map $\tc[\sigma]: \tc[I] \to \tc[J]$ given by $\tc[\sigma](c):= \sigma \circ c \circ \sigma^{-1}$. This defines the species of cycles $\tc$.

\medbreak

If $c \circlearrowleft I$, we write $c = (i_1 \, \,  i_2 \, \, \cdots \, \, i_n)$ if 
\[c: i_1 \mapsto i_2 \mapsto \cdots \mapsto i_n \mapsto i_1,\]
where $\{i_1, i_2, \hdots, i_n\} = I$.

\medbreak

Consider the species of linear orders $\wL$ (see Example \eqref{TwAlgL}). We define two maps $\textsf{lin}: \tc \to \wL$ and $\textsf{cyc}: \wL \to \tc$ given by
\begin{align*}
\textsf{lin}_I: \tc[I] &\to \wL[I] &&,& \textsf{cyc}_I: \wL[I] &\to \tc[I] \\
c &\mapsto \sum_{i \in I}i|c(i)|\cdots|c^{n-1}(i) &&& i_{j_1}|i_{j_2}| \cdots | i_{j_n} &\mapsto (i_{j_1} \, \, i_{j_2} \, \, \cdots \, \, i_{j_n}),
\end{align*}
for all $I=\{i_1, \hdots, i_n\}$.

\medbreak 

If $c \in \tc[I]$, we have $\textsf{cyc}_I(\,i|c(i)|\cdots | c^{n-1}(i)\,)= c$, for every $i \in I$. Therefore, $\textsf{cyc}_I\circ \textsf{lin}_I (c) = |I|\, c$, for all $c \in \tc[I]$.

\

We now describe a NPL-operad structure on $\tc$. Consider the operad $\textsf{As}_+$ on $\wL_+$ (see Example \eqref{OperadAss}), with partial operations $\{\circ_s\}_s$. Given nonempty finite sets $S,T$ and a fixed element $s \in S$, we define a partial composition
\[\rhd_s: \tc[S] \otimes \tc[T] \to \tc[S \sqcup_s T]\]
\begin{equation}
c \rhd_s d:= \textsf{cyc}_{S \sqcup_s T}{\big(}\, s|c(s)|\cdots|c^{|S|-1}(s) \, \circ_s \, \textsf{lin}_T(d) \, {\big)}
\end{equation}
More precisely,
\[c \rhd_s d = \sum_{t \in T} {\big(} \, t \, \, d(t) \, \,\cdots \, \, d^{|T|-1}(t)\,\,c(s) \, \, \cdots \, \, c^{|S|-1}(s)\, {\big)}\]

\begin{proposition}
The species of cycles $\tc$ is a NPL-operad.
\end{proposition}

\begin{proof}


Let $I= S \sqcup T \sqcup R$ and fix distinct elements $s,s' \in S$ and $t \in T$. Consider cycles $c \in \tc[S]$, $d \in \tc[T]$ and $e \in \tc[R]$.

\begin{itemize}
\item Horizontal associativity (A1).
\begin{align*}
(c \rhd_s d)\rhd_{s'}e&= \sum_{t \in T}  {\big(} \, t \, \, d(t) \, \,\cdots \, \, d^{|T|-1}(t)\,\,c(s) \, \, \cdots \, \, c^{|S|-1}(s)\, {\big)} \rhd_{s'} e\\
&= \sum_{t \in T}  {\big(} \, t \, \, d(t) \, \,\cdots \, \, d^{|T|-1}(t)\,\,c(s) \, \, \cdots \, \, s' \, \, \cdots \, \, c^{|S|-1}(s)\, {\big)} \rhd_{s'} e\\
&= \sum_{t \in T}\sum_{r \in R}  {\big(} \, t \, \, d(t) \, \,\cdots \, \, d^{|T|-1}(t)\,\,c(s) \, \, \cdots \, \, r \, \, e(r) \, \, \cdots \, \, e^{|R|-1}(r) \, \, \cdots \, \, c^{|S|-1}(s)\, {\big)}\\
&= \sum_{r \in R}  {\big(} \, s\,\,c(s) \, \, \cdots \, \, r \, \, e(r) \, \, \cdots \, \, e^{|R|-1}(r) \, \, \cdots \, \, c^{|S|-1}(s)\, {\big)} \rhd_s d\\
&= (c \rhd_{s'} e)\rhd_s d.
\end{align*}
\medbreak
\item Nested Pre-Lie axiom (NPL). 

We calculate $(c \rhd_s d) \rhd_t e$ and $c \rhd_s(d \rhd_t e)$ separately: 
\begin{align}
c \rhd_s (d \rhd_t e)  &=  \sum_{r \in R} c \rhd_s  {\big(} \, r \, \, e(r) \, \,\cdots \, \, e^{|R|-1}(r)\,\,d(t) \, \, \cdots \, \, d^{|T|-1}(t)\, {\big)} \nonumber
\end{align}

\[= \sum_{r \in R} \textsf{cyc}_{S \sqcup_s T\sqcup_t R}\left(\textsf{lin}_{T\sqcup_t R}\left( \, r \, \, e(r) \, \,\cdots \, \, e^{|R|-1}(r)\,\,d(t) \, \, \cdots \, \, d^{|T|-1}(t)\, \right)|c(s)|\cdots |c^{|S|-1}(s) \right) \]

\begin{align}
= \sum_{r \in R} \textsf{cyc}_{S \sqcup_s T\sqcup_t R}\left( \, c(s)|\cdots |c^{|S|-1}(s)|\,\textsf{lin}_{T\sqcup_t R}\left( \, r \, \, e(r) \, \,\cdots \, \, e^{|R|-1}(r)\,\,d(t) \, \, \cdots \, \, d^{|T|-1}(t)\, \right) \right). \label{equ1}
\end{align}

\

Also,
\begin{align}
(c \rhd_s d) \rhd_t e  &= \sum_{i \in T}  {\big(} \, i \, \, d(i) \, \,\cdots \, \, d^{|T|-1}(i)\,\,c(s) \, \, \cdots \, \, c^{|S|-1}(s)\, {\big)} \rhd_{t} e \nonumber\\
&= \sum_{i \in T}  {\big(} \, i \, \, d(i) \, \,\cdots \, \, d^{p_i-1}(i) \, \, t \, \, d^{p_i+1}(i) \, \, \cdots \, \, d^{|T|-1}(i)\,\,c(s) \, \, \cdots \, \, c^{|S|-1}(s)\, {\big)} \rhd_{t} e \nonumber\\
&= \sum_{i \in T}  {\big(} \, t \, \, d^{p_i+1}(i) \, \,\cdots \, \, d^{|T|-1}(i)\,\,c(s) \, \, \cdots \, \, c^{|S|-1}(s)\, \, i \, \, d(i) \, \, \cdots \, \, d^{p_i -1}(i) \,{\big)} \rhd_{t} e \nonumber\\
&= \sum_{i \in T} \sum_{r \in R} {\big(} \, r \, \, e(r) \, \, \cdots \, \, e^{|R|-1}(r)  \, \, d^{p_i+1}(i) \, \,\cdots \, \, d^{|T|-1}(i)\,\,c(s) \, \, \cdots \, \, c^{|S|-1}(s)\, \, i \, \, d(i) \, \, \cdots \, \, d^{p_i -1}(i) \,{\big)} \nonumber\\
&=\sum_{i \in T}\textsf{cyc}_{S \sqcup_s T \sqcup_t R} \left(c(s)| \cdots | c^{|S|-1}(s) |\,i\,|d(i)|\cdots | d^{p_i-1}(i)| \, \textsf{lin}_R(e) \,| d^{p_i+1}(i)|\cdots|d^{|T|-1}(i)\,\right), \label{equ2}
\end{align}

where, for every $i \in I$, the integer $0 \leq p_i\leq |T|-1$ is such that $d^{p_i}(i)=t$. Subtracting \eqref{equ1} with \eqref{equ2}, we obtain the following expression for $c \rhd_s (d \rhd_t e) - (c \rhd_s d) \rhd_t e$:
\[ \sum_{r \in R}\sum_{m=0}^{|R|-1} \textsf{cyc}_{S \sqcup_s T \sqcup_t R} \left(c(s)| \cdots | c^{|S|-1}(s) | \, r \, | e(r)| \cdots |e^m(r)| d(t)|\cdots | d^{|T|-1}(t)| e^{m+1}(r)| \cdots | e^{|R|-1}(r) \right).\]

Notice that the expression above is symmetric on $c$ and $d$. Therefore,
\[c \rhd_s (d \rhd_t e) - (c \rhd_s d) \rhd_t e = d \rhd_t (c \rhd_s e) - (d \rhd_t c) \rhd_s e.\]
\end{itemize}
\end{proof}

\

Consider now the species $\mathfrak{S}:= \wE \circ \tc_+$. This is the \emph{species of permutations}. If $I$ is a finite set, then
\[\mathfrak{S}[I]=\mathbb{K}\{f: I \xrightarrow{\cong} I\},\]
where the transport of structure is defined as for $\tc$. If $f: I \xrightarrow{\cong} I$ is a permutation of $I$, then there exists a unique partition $\pi = \{B_1, \hdots, B_k\}\vdash I$ such that $f$ factorizes as a commutative product of cycles $f=c_1\cdots c_k$, with $c_i \in \tc[B_i]$. Given a block $B \in \pi$ and a cycle $c \in \tc[B]$ of $f$, we write $f_{\widehat{c}} \in \mathfrak{S}[I\setminus B]$ for the permutation of $I \setminus B$ obtained by removing the cycle $c$ from the factorization of $f$.

\medbreak

Using Proposition \eqref{PropnplTpnpl}, the species $\mathfrak{S}$ possesses a NPL-operad structure, inherited from the NPL-operad structure $\{\rhd_s\}_s$ of $\tc$. Let $S,T$ be nonempty finite sets, with $s \in S$, and $f \in \mathfrak{S}[S]$, $g \in \mathfrak{S}[T]$. Let $c_s$ be the cycle of $f$ containing the element $s$. By Proposition \eqref{PropnplTpnpl}, the partial composition $\sq_s$ given by
\[\mathfrak{S}[S] \otimes \mathfrak{S}[T] \to \mathfrak{S}[S \sqcup_sT]\]
\begin{equation}
f \,\sq_s \,g:= \sum_{d \text{ cycle of } g}(c_s \rhd_sd)\,f_{\widehat{c_s}}\,g_{\widehat{d}}
\end{equation}
defines a NPL-operad structure on $\mathfrak{S}$.

\
\section{Algebras over a NPL-operad }\label{sect:pol}
Let $V$ be a finite-dimensional vector space (over a field $\mathbb{K}$). For any finite set $A$, note
\begin{equation}\label{calE}
\mathcal{E}(V)[A]:= \mathcal{P}ol (V^{\otimes A}, V)    
\end{equation}

as the space of polynomial functions on $V^{\otimes A}$ with values in $V$.
A function $f: V^{\otimes A} \to V \in \mathcal{E}(V)$[A] can be written as $f=\displaystyle \sum_{i=1}^d f_i e_i$, where $(e_1, \ldots, e_d)$ is a basis of $V$ and $f_i= e_i^* \circ f$, where $(e_1^*, \ldots, e_d^*)$ is the dual basis. We can identify $\mathcal{E}(V)[\{*\}]:= \mathcal{P}ol (V, V)$ with the set of vector fields on $V$, by identifying $V$ with its tangent space at each point, and thus identifying a basis vector $e_i$ with the constant vector field $\partial_i=\frac{\partial}{\partial x_i}$. Any $f\in \mathcal{E}(V)[\{*\}]$ is therefore written as
\[f=\sum_{i=1}^d f_i\partial_i.\]\\

\noindent Let us recall  that $\big(\mathcal{P}ol (V, V), \vartriangleright\!\!\big)$  is pre-Lie algebra, where  $\vartriangleright$ is given by:
\begin{equation}\label{NPLpoly}
\left(\displaystyle \sum_{i=1}^d f_i \partial_i\right)\vartriangleright \left(\displaystyle \sum_{j=1}^d g_j \partial_j\right) =  \displaystyle \sum_{j=1}^d \left(\displaystyle \sum_{i=1}^d f_i (\partial_i g_j)\right) \partial_j,
\end{equation}
for $f=\displaystyle \sum_{i=1}^d f_i e_i \in \mathcal{P}ol (V, V)$ and $g=\displaystyle \sum_{j=1}^d g_j e_j \in \mathcal{P}ol (V, V)$.

\begin{definition}\label{partialNPL}
    Let $B$ and $C$ be two finite sets and $b \in B$. We define a partial composition \[\rhd_{b}:\mathcal{E}(V)[B]\otimes \mathcal{E}(V)[C]\longrightarrow \mathcal{E}(V)[B \sqcup_{b}C]\] 
 as
 \[g \rhd_b f (vw):= \frac 1{|C|}\displaystyle \sum_{k \in C} \left(g_v  \vartriangleright f_{w_{\widehat{k}}}\right)(w_k),\]
 where $w=(w_{k^{'}})_{k^{'}\in C} \in V^{\otimes C }$ and $w_{\widehat{k}}=(w_{k^{'}})_{k^{'}\in C\setminus{\{k\}}}$, and where $g_v \in \mathcal{E}(V)[\{b\}]$ (resp. $f_{w_{\widehat{k}}}\in \mathcal E(V)[\{*\}]$) is defined by $g_v(x)=g(vx)$ for any $v=(v_j)_{j\in B\setminus\{b\}}\in V^{\otimes B\setminus\{b\}}$ (resp. by $f_{w_{\widehat{k}}}(x)=f(w_{\widehat{k}}x)$). 
\end{definition}
\begin{theorem}
   $ \big(\mathcal{E}(V), \rhd\big)$ is an NPL-operad.
\end{theorem}
\begin{proof}
\phantom{a}
\begin{itemize}
    \item Parallel associativity:  let $A$, $B$, and $C$ be three finite sets and let $h\in \mathcal{E}(V)[A]$,  $g \in \mathcal{E}(V)[B]$ and $f \in \mathcal{E}(V)[C]$. Let $a, a' \in A$, let $u \in V^{\otimes A \setminus{\{a, a'\}}}$, let $v \in V^{\otimes B }$, and let $w\in V^{\otimes C }$.
    \begin{eqnarray*}
        (h \rhd_a g) \rhd_{a'} f(uvw)&=&\frac{1}{|C|} \displaystyle \sum_{k \in C}\left((h_u \rhd_a  g)_{uv} \vartriangleright f_{w_{\widehat{k}}}\right)(w_k)\\
        &=& \frac{1}{|C|}\displaystyle \sum_{k \in C} \Bigg( \big( x \mapsto h \rhd_a g (xuv)\big)\vartriangleright f_{w_{\widehat{k}}}\Bigg)(w_k)\\
        &=& \frac{1}{|B||C|}\displaystyle \sum_{k \in C, \,j \in B } \Bigg( \left( x \mapsto h_{xu} \vartriangleright g_{v_{\widehat{j}}} (v_j)\right)\vartriangleright f_{w_{\widehat{k}}}\Bigg)(w_k).\\
        & & \qquad \qquad \qquad 
    \end{eqnarray*}
    From
    \begin{eqnarray*}
        h_{xu} \vartriangleright g_{v_{\widehat{j}}} (v_j)&=& \displaystyle \sum_{p}\displaystyle \sum_{r} h_{xu}^r \partial_r g_{v_{\widehat{j}}}^p \partial_p (v_j)\\
        &=& \displaystyle \sum_{p}\underbrace{\displaystyle \sum_{r} h_{xu}^r(v_j) \left( \partial_r g_{v_{\widehat{j}}}^p \right) (v_j)}_{H_p}\partial_p \\
        &:=& \displaystyle \sum_{p} H^p \partial_p.
    \end{eqnarray*}
    we get
    \begin{eqnarray*}
        (h \rhd_a g) \rhd_{a'} f(uvw)&=&\frac{1}{|B||C|}\displaystyle \sum_{k \in C, \,j \in B } \Bigg(\displaystyle \sum_{p} H^p \partial_p \Bigg) \vartriangleright \Bigg(\displaystyle \sum_{q} f_{w_{\widehat{k}}}^q \partial_q \Bigg) (w_k)\\
        &=& \frac{1}{|B||C|}\displaystyle \sum_{k \in C, \,j \in B } \Bigg(\displaystyle \sum_{q}\displaystyle \sum_{p} H^p \left(\partial_p f_{w_{\widehat{k}}}^q \right)\partial_q \Bigg) (w_k)\\
        &=& \frac{1}{|B||C|}\displaystyle \sum_{k \in C, \,j \in B }\displaystyle \sum_{q}\displaystyle \sum_{p} \displaystyle \sum_{r} h_{u}^r(v_j w_k) \left( \partial_r g_{v_{\widehat{j}}}^p \right) (v_j)\left(\partial_p f_{w_{\widehat{k}}}^q \right)(w_k)\partial_q.
    \end{eqnarray*}
    Here it is clear that $g_{v_{\widehat{j}}}$ and $f_{w_{\widehat{k}}}$ play a symmetric role, which implies 
    \[(h \rhd_a g) \rhd_{a'} f(uvw)=(h \rhd_{a'} f) \rhd_{a} g(uwv),\]
    hence (A1) is proved.\\
    
    \item  The nested pre-Lie axiom (NPL) is proven as follows: let $A$, $B$, and $C$ be three finite sets and let $h\in \mathcal{E}(V)[A]$,  $g \in \mathcal{E}(V)[B]$ and $f \in \mathcal{E}(V)[C]$. Let $u \in V^{\otimes A \setminus{\{a\}}}$, $v \in V^{\otimes B \setminus{\{b\}}}$ and $w\in V^{\otimes C }$. We can compute
    \begin{align*}
        h \rhd_a(g \rhd_b f)(uvw)&=\frac{1}{|B|+|C|-1} \left(\displaystyle \sum_{j \in B \setminus{\{b\}}}\left(h_u \vartriangleright (g \rhd_b f)_{v_{\widehat{j}}w}\right)(v_j)+\displaystyle \sum_{k \in C} \left(h_u \vartriangleright (g \rhd_b f)_{vw_{\widehat{k}}}\right)(w_k)\right)\\
        &= \frac{1}{|B|+|C|-1}\Bigg\{\displaystyle \sum_{j \in B \setminus{\{b\}}}\Big(h_u \vartriangleright \left(x \mapsto g \rhd_b f\right)(xv_{\widehat{j}}w)\Big)(v_j)\\
        &\phantom{=}+\displaystyle \sum_{k \in C}\Big(h_u \vartriangleright \left(x \mapsto g \rhd_b f\right)(vw_{\widehat{k}}x)\Big)(w_k)\Bigg\}\\
        &= \frac{1}{(|B|+|C|-1)|C|}\Bigg\{\displaystyle \sum_{j \in B \setminus{\{b\}},k\in C}\bigg(h_u \vartriangleright \left(x \mapsto g_{xv_{\widehat{j}}} \vartriangleright f_{w_{\widehat{k}}}\right)(w_k)\bigg)(v_j)\\
        & \phantom{=}+\displaystyle \sum_{k, k^{'} \in C, k \ne k^{'}}\bigg(h_u \vartriangleright \Big(x \mapsto g_v \rhd f_{w_{\widehat{k},\widehat{k^{'}} }x}\Big)(w_{k^{'}})\bigg)(w_k) \\
        & \phantom{=} + \displaystyle \sum_{k \in C}\bigg(h_u \vartriangleright \Big(x \mapsto g_v \vartriangleright f_{w_{\widehat{k}}}\Big)(x)\bigg)(w_k)\Bigg\} \\
        &=\frac{1}{|C|}\displaystyle \sum_{k \in C}\bigg(h_u \vartriangleright \left(g_v \vartriangleright f_{w_{\widehat{k}}}\right)\bigg)(w_k).
    \end{align*}
    \noindent On the other hand,
    \begin{eqnarray*}
        (h \rhd_a g) \rhd_b f(uvw)&=& \frac{1}{|C|}\displaystyle \sum_{k \in C}\left((h \rhd_a  g)_{uv} \vartriangleright f_{w_{\widehat{k}}}\right)(w_k)\\
        &=& \frac{1}{|C|}\displaystyle \sum_{k \in C} \bigg( \left( x \mapsto h \rhd_a g (uvx)\right)\vartriangleright f_{w_{\widehat{k}}}\bigg)(w_k)\\
        &=&\frac{1}{|B||C|} \Bigg\{\displaystyle \sum_{k \in C, \,j \in B \setminus{\{b\}} } \bigg( \left( x \mapsto h_u \vartriangleright g_{v_{\widehat{j}}x} (v_j)\right)\vartriangleright f_{w_{\widehat{k}}}\bigg)(w_k)\\
        &&+ \displaystyle \sum_{k \in C}\bigg(\left(h_u \vartriangleright g_v\right) \vartriangleright f_{w_{\widehat{k}}}\bigg)(w_k)\Bigg\}\\
        &=&\frac{1}{|C|}\sum_{k \in C}\bigg(\left(h_u \vartriangleright g_v\right) \vartriangleright f_{w_{\widehat{k}}}\bigg)(w_k).
    \end{eqnarray*}
    We therefore have
    \begin{eqnarray*}
    &&\Big(h \rhd_a(g \rhd_b f)-(h \rhd_a g) \rhd_b f-g \rhd_b(h \rhd_a f)+(g \rhd_b h) \rhd_a f\Big)(uvw)\\
    &&=\frac{1}{|C|}\displaystyle \sum_{k \in C}\bigg(h_u \vartriangleright \left(g_v \vartriangleright f_{w_{\widehat{k}}}\right)-\left(h_u \vartriangleright g_v \right)\vartriangleright f_{w_{\widehat{k}}}-g_v \vartriangleright \left(h_u \vartriangleright f_{w_{\widehat{k}}}\right)+\left(g_v \vartriangleright h_u \right)\vartriangleright f_{w_{\widehat{k}}}\bigg)(w_k)\\
    &&=0
    \end{eqnarray*}
    \noindent from the left pre-Lie identity for $\rhd$.
    \end{itemize}
    \end{proof}

Recall that the subspecies $\mop{End}(V)\subseteq \mathcal E(V)$, defined by $\mop{End}(V)[B]:=\mathcal L(V^{\otimes B},V)$ is an operad. The partial compositions are defined for any finite sets $B,C$ and any $b\in B$ by
\[g\circ_b f(vw):=g\big(vf(w)\big)\]
for any $g\in\mathcal L(V^{\otimes B},V)$ and $f\in\mathcal L(V^{\otimes C},V)$. Here $v\in V^{\otimes B\setminus \{b\}}$ and $w\in V^{\otimes C}$, so that the word $vf(w)$ belongs to $V^{\otimes (B\setminus\{b\})\sqcup \{*\}}\simeq V^{\otimes B}$.

\begin{lem}\label{NPLinclusion}
The inclusion $\mop{End}(V) \hookrightarrow \mathcal{E}(V)$ is a morphism of NPL-operads.
\end{lem}
\begin{proof}
Let $A,B$ be two finite sets and let $a\in A$. It is well-known, and easily seen, that for any $\alpha,\beta\in\mathcal L(V,V)$ the equality $\alpha\rhd \beta=\alpha\circ \beta$ holds. It is therefore immediate from Definition \ref{partialNPL} that, for any $g\in\mathcal L(V^{\otimes B},V)$ and $f\in\mathcal L(V^{\otimes C},V)$, the equality
\[g\rhd_b f=g\circ_b f\]
holds, where $\circ_b$ stands for the usual partial composition in $\mop{End}(V)$.
\end{proof}

\begin{definition}
    Let $V$ be a vector space and $\mathcal{P}$ be an NPL-operad.
    We say that $V$ is a $\mathcal{P}$-algebra if there exists an NPL operad morphism  $\Psi: \mathcal{P}\to \mathcal{E}(V)$.
\end{definition}
\begin{proposition}\label{NPLalgebra}
    An algebra $V$ over an operad $\tp$ is also an algebra over $\tp$ seen as an NPL operad.
\end{proposition}
\begin{proof}
Let us recall that $V$ is an algebra over an operad $\tp$ if and only if there exists an operad morphism  $\Phi: \tp\to \text{End}(V)$. Proposition \ref{NPLalgebra} is therefore obtained from Lemma \ref{NPLinclusion} by composing this operad morphism with the canonical injection
\[\xymatrix{
\tp \ar[r]^-{\Phi} \ar[rd]_-{\Psi} &
\text{End}(V) \ar@{^{(}->}[d]\\
&\mathcal{E}(V)}\]
\end{proof}

\

\bibliographystyle{siam}
\bibliography{Bibliography}

\end{document}